\documentclass[10pt]{amsart}

\usepackage{verbatim}

\usepackage{latexsym}
\usepackage{enumerate}
\usepackage{amsmath,amssymb,amsthm,amsfonts,latexsym}
\usepackage{amsmath,amscd}
\usepackage{pstricks, pst-node, pst-text, pst-3d}
\usepackage[bookmarks]{hyperref}
\hypersetup{backref, colorlinks=true}

\let\=\partial

\newtheorem {prop}{Proposition} [section] 
\newtheorem {thm}[prop]{Theorem}
 \newtheorem {cor}[prop]{Corollary}
\newtheorem{lem}[prop]{Lemma}

\theoremstyle{definition}
 \newtheorem {rk}[prop]{Remark}
\newtheorem {df}[prop]{Definition}
\newtheorem {dfs}[prop]{Definitions}

\newtheorem {notation}[prop]{Notation}
\newtheorem {notations}[prop]{Notations}

\newcommand{\Q} {\mathbb{Q}}

\newcommand{\G} {\mathbb{G}}
\newcommand{\R} {\mathbb{R}}

\newcommand{\N} {\mathbb{N}}
\newcommand{\C} {\mathcal{C}}

\newcommand{\eps}{\varepsilon}

\newcommand{\A}{\mathcal{A}}
\newcommand{\B}{\mathcal{B}}

\newcommand{\I}{\mathcal{I}}

\newcommand{\U}{\mathcal{U}}
\newcommand{\F}{\mathcal{F}}

\newcommand{\hk}{(H_k;\lambda_k)_{1 \leq k \leq b}}
\newcommand{\pa}{\partial}



\author{Leonid Shartser and Guillaume Valette}

\address
{Instytut Matematyczny PAN, ul. Sw. Tomasaza 30, 31-027 Krakow, Poland}

\address
{Department of Mathematics, University of Toronto, 40 St. George
st, Toronto, ON, Canada M5S 2E4 }

\email{gvalette@math.toronto.edu}
\email{gvalette@impan.pl}
\email{shartl@math.toronto.edu}

\keywords{Homology theory, metric invariants, definable sets, De Rham Theory}

\thanks{}

\subjclass{14P10, 14P25, 55N20}

\begin{document}
\title{De Rham theorem for $L^\infty$ forms and homology on singular spaces.}

\begin{abstract}
We introduce smooth $L^{\infty}$ differential forms on a singular 
(semialgebraic) set $X$ in $\R^n$. Roughly speaking, a smooth $L^{\infty}$ differential form is a certain class of equivalence of 'stratified forms', that is, a collection of smooth forms on disjoint smooth subsets (stratification) of $X$ with matching tangential components on the adjacent strata and bounded size (in the metric induced from $\R^n$).

We identify the singular homology of $X$ as the homology of the chain complex 
generated by semialgebraic singular simplices, i.e. continuous 
semialgebraic maps from the standard simplices into $X$. 
Singular cohomology of $X$ is defined as the homology of the Hom dual 
to the chain complex of the singular chains.
Finally, we prove a De Rham type theorem establishing a natural isomorphism 
between the singular cohomology and the cohomology of smooth $L^{\infty}$ forms.

\end{abstract}
\maketitle

\section{Introduction}

In a recent paper J.-P Brasselet and  M.J.Pflaum [BPf] proved a De Rham type theorem for Whitney functions.
Namely, they showed that the cohomology of the complex of Whitney differential forms naturally coincides
with the singular cohomology of a subanalytic set. 
 
In the present work we study the cohomology of the cochain complex of the, so called, 
smooth $L^\infty$ forms (see Definition \ref{l_infty_class}), which is intrinsically defined for the singular spaces in question.

The singular spaces considered in this paper are semialgebraic subsets of $\R^n$. Let $X\subset\R^n$ be a semialgebraic set. A stratification of $X$ is a collection of smooth (semialgebraic) locally closed manifolds
 (strata) in $X$ such that the boundary of each stratum is a union of strata.

We introduce a version of De Rham theory for
$L^{\infty}$ differential forms on $X$ for which we prove a De Rham type theorem.
An $L^{\infty}$ differential form on $X$ is, roughly speaking, the data of a stratification
 $\Sigma$ of $X$ and
a collection of smooth forms on the nonsingular strata such that the tangential
components of the forms on the adjacent strata match and the size of the form
is bounded (in the metric induced from $\R^n$). We consider two $L^\infty$ forms to be equivalent
if there exists a stratification of $X$ on which the restrictions of the two forms coincide.
The exterior derivative of such forms
is defined by the exterior derivatives of their restrictions to strata. The $L^{\infty}$ forms with $L^{\infty}$ exterior derivatives form
a cochain complex.
The main theorem of this paper is a De Rham type theorem (Theorem \ref{De Rham}). 
Namely, we prove that the map 
$$ \omega \stackrel{\psi}{\mapsto} "c \mapsto \int_c \omega "\ ,$$
from the complex of $L^{\infty} $ forms to the space of semialgebraic singular cochains,
is a natural map of chain complexes that induces an isomorphism on cohomology.
\\

All of the considered subsets and mappings will be assumed to be
semialgebraic, except the differential forms which will be
smooth on each stratum. \\

The paper is organized as follows. 
In Section \ref{sec_main_results} we set the notations and define the objects of study.
In Section \ref{sec_stokes} we prove the Stokes' formula for $L^\infty$ forms, that is, we prove that $\psi$ is a chain-map. 
In Section \ref{sec_de_rham} we prove Theorem \ref{De Rham}. 
One of the key ingredients in the proof of De Rham Theorem is the $L^\infty$ version of Poincar\'{e} lemma, i.e. locally closed $L^\infty$ forms are exact.  
An essential tool in proving our Poincar\'e lemma is a Lipschitz strong deformation retraction
that preserves a certain stratification.
In Section \ref{sec_retract} we construct such deformation retractions. 
We prove that for any point $p\in X$ there exists a  set $N\subset X$ that contains $p$ with $\dim N<\dim X$ and a Lipschitz strong deformation retraction to $N$ of a neighborhood $U$ of $N$ preserving a given stratification of $U$ (see Theorem \ref{retraction}).
In Section \ref{smoothing} we prove the $L^\infty$ Poincar\'e Lemma.

\medskip

\vspace*{4mm}
\noindent \textbf{Acknowledgment. }
We would like to thank Pierre Milman for raising the question, helpful
discussions and important comments.
\vspace*{2mm}

\section{The Main Results}\label{sec_main_results}

We begin by introducing the basic notations and definitions.
All of the considered subsets and mappings will be assumed to be
semialgebraic unless explicitly specified otherwise except the differential forms which will be
smooth on each stratum. 

\begin{notations}
Suppose that $f:X\to Y$ is a map of topological spaces. We write $f(x)\to y$ as $x\to a$ to denote
$\lim_{x\to a} f(x)=y$. Suppose that $S\subset X$. The graph of $f$ over $S$ is 
$$\Gamma_f(S):=\{(x,y)\in S\times Y : y=f(x) \}$$
Denote by $\G^k_n$ the Grassmaniann manifold of $k$ dimensional subspaces of $\R^n$. 

\end{notations}

\begin{dfs}

Suppose that  $S\subset\R^n$. The {\bf closure} of $S$ is denoted by $\overline S$ and the 
{\bf topological boundary} of $S$ is defined by $\pa_{T} S:=\overline S - \overline{(\R^n-S)}$. If $S$ is a manifold then the {\bf boundary} of $S$ is $\pa S:=\overline S - S$. Note that $\pa S$ is a subset of $\R^n$ of dimension at most $\dim S - 1$ (not necessarily smooth).

A {\bf stratified} space $(X,\Sigma)$ is a set $X\subset\R^n$ together with a partition (stratification) $\Sigma$ of $X$
into locally closed manifolds (strata) such that the boundary of each stratum is a union of strata in $\Sigma$.
If $S$ and $S'$ are two strata in $\Sigma$ such that $S'\subset\partial S$ then we write $S'\leq S$. 
Denote by $\Sigma^k$ the collection of all strata in $\Sigma$ of dimension $k$, by $\Sigma^{(k)}$ the collection of all strata up to dimension $k$ and by $|\Sigma|$ the union of all strata in $\Sigma$.
A {\bf refinement} of $\Sigma$ is a stratification $\Sigma'$
such that each stratum of $\Sigma$ is a union of strata of $\Sigma'$. We then write $\Sigma'\prec\Sigma$. If $f:X\to Y$ is a map and $\Sigma$ is a stratification of $X$ then we
write $f(\Sigma)$ to denote the collection of sets $\{f(S):S\in\Sigma\}$.

The {\bf tangent bundle} of $(X,\Sigma)$ is 
$$TX:=\cup_{x\in X}\{x\}\times T_x X\subset \R^n\times\R^n, $$ 
with the subspace topology, where $T_x X:=T_x S$ , if  $x\in S\in\Sigma$. Similarly, 
define the exterior product of the tangent bundle 
as 
$$ \wedge^k TX:=\cup_{x\in X} \{x\}\times\wedge^k T_x X\subset \wedge^k(T\R^n) . $$

A {\bf stratified (resp. smooth) $k$-form} on $X$ is a pair $(\omega,\Sigma)$, where $\Sigma$ is 
a stratification of $X$ and $\omega=(\omega_S)_{S \in\Sigma}$ where
$\omega_S$ is a continuous (resp. smooth) differential $k$-form on $S\in\Sigma$, such that
the graph of $\omega : \wedge^k TX \to \R$, $(x,\xi)\mapsto \omega(x;\xi)$ is closed in $\wedge^k TX\times\R$.

The {\bf exterior derivative} of a smooth stratified form $(\omega,\Sigma)$ is defined as $(d\omega_S)_{S \in\Sigma}$ and denoted by $(d\omega,\Sigma)$.

The {\bf weak exterior derivative}
of a $k$-form $\omega$ in $\R^n$ is a $(k+1)$-form $\overline d\omega$ in $\R^n$ such that for any smooth $(n-k-1)$-form $\varphi$ in $\R^n$, with compact support, we have
\begin{equation}\label{weak_ext_der}
 \int_{\R^n} \overline d \omega\wedge\varphi = (-1)^{k+1}\int_{\R^n} \omega\wedge d\varphi\ .
\end{equation}

Suppose that $S$ is a manifold and $\omega$ is a $k$-form on $S$. 
A $(k+1)$-form $\overline d \omega$ on $S$ is called the weak exterior derivative of $\omega$ 
if for every $p\in S$ there exists an open neighborhood $U$ of $p$ and a diffeomorphism $\phi:\R^n\to U$ 
such that (\ref{weak_ext_der}) holds with $\omega$ replaced by  $\phi^*\omega$ and 
$\overline d\omega$ replaced by $\phi^*\overline d \omega$.

The weak exterior derivative of a stratified form $(\omega,\Sigma)$ is defined as $(\overline d\omega_S)_{S \in\Sigma}$ and denoted 
by $(\overline d\omega,\Sigma)$. 

When there is no confusion possible, we write $\omega$, $d\omega$ and $\overline d\omega$ instead of 
$(\omega,\Sigma)$, $(d\omega,\Sigma)$ and $(\overline d \omega,\Sigma)$.

If $(\omega,\Sigma)$ is a stratified form then a refinement $\Sigma'\prec\Sigma$ induces
a stratified form $(\omega,\Sigma')$ in a canonical way.

Furthermore, given a stratified $k$-form $(\omega,\Sigma)$ on $X$,
we say that $\omega$ is {\bf bounded} if $|\omega_{S}|\leq
C$ for all $S\in\Sigma$ and some positive constant $C$, where
the norm $|\cdot|$ is induced from $\R^n$. 
The set of all stratified and bounded $k$-forms on $X$ 
with stratified exterior derivatives is denoted by $\widetilde \Omega^k(X)$.

\end{dfs}

\begin{notation}
Let $V$ be a vector space, $v_i$, $i=1,\dots ,m$ be some elements of $V$ and 
$I=(i_1,\dots,i_k)$ be a multi-index. 
We denote by $v_I$ the element $v_{i_1}\wedge\dots\wedge v_{i_k}\in \wedge^k(V)$.
Sometimes we will use superscripts instead of subscripts.
\end{notation}

Bounded stratified forms have the following continuity property.
\begin{lem}\label{cont_property}
If $\omega=(\omega,\Sigma)$ is a stratified and bounded $k$-form then for any pair 
$(S,S')\in\Sigma\times\Sigma$ with 
$S'\leq S$, any sequence
of points $p_{n}\in S $ with $\lim_{n\to\infty} p_{n}=p\in S'$
and any sequence of multivectors
$\xi_{n}\in \wedge^k T_{p_{n}} S$ with
$$\lim_{n\to\infty}\xi_{n}=\xi \in {\wedge}^k T_{p} S'$$ we have
$$\lim_{n\to\infty} \omega(p_n;\xi_{n})= \omega(p;\xi). $$
\end{lem}
\begin{proof}
Since $\omega$ is bounded there exists a subsequence $(p_{n_j},\xi_{n_j})$
and a real number $a$ such that 
$$\lim_{j\to\infty} \omega(p_{n_j};\xi_{n_j})= a. $$
But since the graph of $\omega$ is closed, it follows that $a=\omega(p,\xi)$.
\end{proof}

\begin{prop}\label{wedge_product}
Suppose that $(X,\Sigma)$ is a stratified set in $\R^n$, $(\alpha,\Sigma)$ and $(\beta,\Sigma)$
are smooth and bounded stratified forms of degree $k$ and $l$ respectively. Then
$\alpha\wedge\beta$ is a smooth and bounded stratified form of degree $k+l$. 
\end{prop}
\begin{proof}
First let us recall that if $S$ is a smooth manifold, $v_1,\dots,v_{k+l}$ are some vectors in $T_x{S}$, $\alpha\in \Omega^{k}(S)$ and $\beta\in \Omega^{l}(S)$ then we have the following 
representation formula
\begin{eqnarray}\label{wedge_rep_form}
&&\\ \nonumber
(\alpha\wedge\beta)(x;v)   = 
\sum_{\sigma} sgn(\sigma) \alpha(x;v_{\sigma(1)}\wedge\dots\wedge v_{\sigma(k)})
\beta(x; v_{\sigma(k+1)}\wedge\dots\wedge v_{\sigma(l+k)})
\end{eqnarray}
where $\sigma$ varies over all permutations of $\{1,\dots,k+l\}$, 
$sgn(\sigma)$ is $1$ if $\sigma$ is even and $(-1)$ if $\sigma$ is odd and
$v=v_1\wedge\dots\wedge v_{k+l}$.
It follows from (\ref{wedge_rep_form}) that there exists an universal constant $C > 0$ such that
$$ |\alpha\wedge\beta|\leq C|\alpha||\beta| .$$
%
%
%
Next we show that the graph of $\alpha\wedge\beta$ is closed.
Observe that if $\xi_1\wedge\dots\wedge\xi_m\in \wedge^{m}(T_x X)$ is of size $1$, then  
there exists an orthonormal set $\{v_1,\dots,v_m \}$ in $T_x X$  such that
$\xi_1\wedge\dots\wedge\xi_m=v_1\wedge\dots\wedge v_m$.

Let $S$ and $S'$ be two strata of $\Sigma$  such that $S'\leq S$.
Suppose that $(x_m;\xi_m)\in \wedge^{k+l}(TS)$ is a sequence that 
tends to $(x;\xi)\in \wedge^{k+l}(TS')$. We have to show that if $(\alpha\wedge\beta)(x_m;\xi_m)$ is convergent then the limit is $(\alpha\wedge\beta)(x;\xi)$.

First assume that $\xi_m=A_m\xi^1_m\wedge\dots\wedge\xi^{k+l}_m$. We may assume, by the observation above, that $\{ \xi^j_m: j=1,\dots,k+l\}$ is an orthonormal set for each $m\in\N$. There exists an
increasing sequence of natural numbers $m_t$, $t\in \N$
such that $\xi^j_{m_t}\to \xi^j$ and $A_{m_t}\to A$ as $t\to\infty$ for $j=1,\dots,k+l$. 
By uniqueness of the limit we clearly have $\xi=A\xi^1\wedge\dots\wedge\xi^{k+l}$.
We may assume without loss of generality that $m_t=t$. 
Using formula (\ref{wedge_rep_form}) we obtain
\begin{eqnarray*}
(\alpha\wedge\beta)(x_m;\xi_m) &=&
A_m\sum_{\sigma} sgn(\sigma) \alpha(x_m;\xi_m^{\sigma(1)}\wedge\dots\wedge\xi_m^{\sigma(k)})
\times\nonumber\\
&\times&\beta(x_m;\xi_m^{\sigma(k+1)}\wedge\dots\wedge\xi_m^{\sigma(l+k)}).
\nonumber
\end{eqnarray*}
By letting $m\to\infty$ we obtain 
\begin{eqnarray*}
\lim_{m\to\infty}(\alpha\wedge\beta)(x_m;\xi_m) &=&
A\sum_{\sigma} sgn(\sigma) \alpha(x;\xi^{\sigma(1)}\wedge\dots\wedge\xi^{\sigma(k)})
\times\nonumber\\
&\times&\beta(x;\xi^{\sigma(k+1)}\wedge\dots\wedge\xi^{\sigma(l+k)})
\\&=&
(\alpha\wedge\beta)(x;\xi).
\nonumber
\end{eqnarray*}
For the general case, suppose that $\xi_m=\sum_I a^I_m \xi^I_m $ with $\{\xi^i_m: i=1,\dots,\dim T_{x_m} X) \}$ an orthonormal set for all $m\in \N$. By choosing a subsequence, similarly 
to the argument above,
we may assume that $a^I_m\to a^I$ and $\xi^I_m\to\xi^I$ as $m\to\infty$ for each $I=(i_1,\dots,i_{k+l})$.
Thus we obtain
\begin{eqnarray*}
\lim_{m\to\infty}(\alpha\wedge\beta)(x_m;\xi_m) 
&=&
\lim_{m\to\infty} \sum_I a_m^I (\alpha\wedge\beta)(x_m;\xi^I_m) 
\nonumber  \\ &=&
\sum_I a^I (\alpha\wedge\beta)(x;\xi^I) 
\\&=&  (\alpha\wedge\beta)(x;\xi) .
\nonumber
\end{eqnarray*}

\end{proof}

\begin{df}\label{semi_diff_strat}
Let $(X,\Sigma)$ and $(Y,\Sigma')$ be two stratified sets in $\R^n$. A continuous 
map $h: X \to Y$ is {\bf semi-differentiable} (see [MT]) if
\begin{enumerate}
\item
For any $S\in \Sigma$ there is a stratum $S'\in\Sigma'$ such that $h$ maps $S$ into $S'$ in a smooth way.
\item
The differentials $dh_i$, $i=1,\dots,n$ of the components of $h$ are smooth stratified 
forms on $(X,\Sigma)$.
\end{enumerate}
\end{df}

\begin{rk}
Semi-differentiability implies Lipschitzness with respect to the inner metric.
\end{rk}
In the following proposition we prove that composition of semi-differential maps is
a semi-differential map.
\begin{prop}\label{comp_semi_diff}
Suppose that $(X,\Sigma_X), (Y,\Sigma_Y)$ and $(Z,\Sigma_Z)$ are stratified subsets of $\R^n$. Let $f:X\to Y$ be a semi-differentiable map with respect to $\Sigma_X$ and $g:Y\to Z$ be a semi-differentiable map with respect to $\Sigma_Y$. Then, 
the composition $g\circ f:X\to Z$ is semi-differentiable with respect to $\Sigma_X$ and $\Sigma_Z$.
\end{prop}
\begin{proof}
The verification of the first part of Definition \ref{semi_diff_strat} is obvious, so we go directly to the proof of the second part.

Suppose that $S'\leq S$ are two strata in $\Sigma_X$. Let $(x_n;\xi_n)\in T_{x_n} S$ be 
a sequence converging to $(x;\xi)\in T_{x} S'$. 
Note that by the chain rule we have 
$$ d(g\circ f)(x_n;\xi_n) = dg (f(x_n);df(x_n;\xi_n)).$$
Since $f$ is semi-differentiable $(f(x_n);df(x_n;\xi_n))$ converges to $(f(x),df(x;\xi))$
and since $g$ is semi-differentiable, it follows that  $(g(f(x_n)); dg (f(x_n);df(x_n;\xi_n)))$ converges to
$(g(f(x)),dg(f(x);df(x;\xi)))$.
\end{proof}

\subsection{$\mathcal L^\infty$ Category}
We introduce here the category '$\mathcal L ^\infty$' in which we work.
The objects of $\mathcal L ^\infty$ are sets. A morphism, or an $L^\infty$ map, between two objects $X$ and $Y$ is a map $f : X \to Y$ such that there exists a stratification $\Sigma_X$ of $X$ and a stratification
$\Sigma_Y$ of $Y$ with respect to which $f$ is a semi-differentiable map. 
It follows from Lemma \ref{comp_semi_diff} that composition of morphisms is
a morphism. Now it is clear that the latter indeed defines a category.\\


\begin{prop}\label{lip_infty}
Let $X\subset\R^n$, $Y\subset\R^m$ then the map $f:X\to Y$ is Lipschitz if and only if $f$ is an $L^\infty$ morphism.
\end{prop} 
The proof of this proposition is in Section 4.

Now we can define the notion of an $L^\infty$ form that fits well into the setting of $\mathcal L ^\infty$ category.

\begin{df}\label{l_infty_class}
Let $X$ be a set and define an equivalence relation on $\widetilde \Omega^k(X)$ 
$$ (\omega,\Sigma) \approx (\omega',\Sigma') $$
if there exists a stratification $\Sigma''$ that refines both $\Sigma$ and $\Sigma'$ such that
$(\omega|_{\Sigma''},\Sigma'')=(\omega'|_{\Sigma''},\Sigma'')$.
Denote by $\Omega_\infty^k (X) $ the classes of equivalence of $'\approx'$. An element of 
$\Omega_\infty^k (X)$ is called an {\bf $\bf L^\infty$ form}.
\end{df}

\begin{rk}
The exterior algebra structure is defined on $\Omega^{\bullet}_\infty (X):=\cup_k \Omega_\infty^k(X)$ in a natural way. The sum of two $L^\infty$ forms $\omega$ and $\omega'$ is an $L^\infty$ form $\omega''$ that can be constructed as follows. If $(\omega,\Sigma)$ and $(\omega,\Sigma')$ represent $\omega$ and $\omega'$ then $(\omega+\omega',\Sigma'')$ represents $\omega''$ where $\Sigma''$ is any stratification refining both $\Sigma$ and $\Sigma'$. The exterior product is defined in a similar fashion.
\end{rk}

\subsection*{Pull backs of $L^\infty$ forms}
\begin{prop} \label{prop_pull_back}
Let $f: X \to Y$ be an $L^\infty$ map and $\omega\in\Omega^k_\infty(Y)$. 
For any stratification $\Sigma_X$ of $X$ and $\Sigma_Y$ of $Y$ for which
$f$ is semi-differentiable and $(\omega,\Sigma_Y)$ is a stratified form,
the form defined by $f|_S^* \omega$ on each $S\in\Sigma_X$ defines a unique $L^\infty$ form
which is called {\bf the  pullback of $\omega$} by $f$ and denoted by  $f^*\omega$.
\end{prop}

\begin{proof}
We have to show that $f^*\omega$ is well defined and $f^*\omega$ is
an $L^\infty$ form. Let $\Sigma_X$ and $\Sigma_Y$ be stratifications of $X$ and $Y$ respectively such that \\

(*) $f$ is semi-differentiable with respect to $\Sigma_X$ and $\Sigma_Y$ and $(\omega,\Sigma_Y)$ is a stratified form.\\

We have to show that the graph of $f^*\omega$ is closed and the class of $\eta:=(f^*\omega,\Sigma_X)$ is independent of the choices of $\Sigma_X$ and $\Sigma_Y$. 
First we show that the graph of $f^*\omega$ is closed. Suppose that $X\subset \R^l$ and $Y\subset \R^m$. Let $S'\leq S$ be two strata in $\Sigma_X$ and $(x_n;\xi_n)\in \wedge^k(TS)$ be a sequence converging to  $(x;\xi)\in\wedge^k(TS')$.
Suppose that $f|_S^*\omega(x_n;\xi_n)$ is convergent. 
We may assume, by possibly choosing a subsequence, that $(f(x_n),df|_S(x_n;\xi_n)\to (f(x),\xi)$. 
Since $f$ is semi-differentiable, it follows that 
$f|_S^*\omega(x_n;\xi_n)\to f|_{S'}^*\omega(x;\xi)$ as \mbox{$n\to\infty$}.
%
%
To see that $f^*\omega$ is independent of the chosen stratifications, let $\Sigma'_X$ and $\Sigma'_Y$ be another stratifications of $X$ and $Y$ satisfying (*) and denote by $\eta':=(f^*\omega,\Sigma'_X)$  the form obtained from $f|_S^*\omega$ for any $S\in \Sigma'_X$.
Let $\Sigma''_X$  be a common refinement of $\Sigma_X$ and $\Sigma'_X$, and  let $\Sigma''_Y$ be a common refinement of $\Sigma_Y$ and $\Sigma'_Y$. Let $\eta'':=(f^*\omega,\Sigma''_X)$ then clearly $\eta\approx\eta''\approx\eta'$.
\end{proof}

\subsection*{Integration of $L^\infty$ forms}
{$\ $}\vspace*{2mm}$\ $\\
Let $X\subset\R^n$ and $A\subset\R^k$ be $k$-dimensional oriented compact semialgebraic sub-manifold with corners of $\R^k$, where $k\leq n$ and the orientation of $A$ is induced by the standard orientation of $\R^k$. Let $\sigma:A\to X$ be a map. We want to define the integral of an $L^\infty$ $k$-form $\omega$ over $\sigma$.
Let $\Sigma_X$ and $\Sigma_A$ be stratifications of $X$ and $A$ respectively such that $(\sigma^*\omega,\Sigma_A)$ is a stratified form.
The integral of $\omega$ over $\sigma$ is defined by 
\begin{equation} \label{def_int}
\int_\sigma \omega := \sum_{S\in\Sigma_A^k} \int_S \sigma_S^*\omega\ .  
\end{equation}

\begin{prop}
Let $X$, $\omega$, $A$ and $\sigma$ be as in the above paragraph then the integral of $\omega$ over $\sigma$, as defined
in (\ref{def_int}), is well defined i.e. independent of the stratifications.
\end{prop}
\begin{proof}
It is enough to prove that for any $\Sigma'_A\prec\Sigma_A$ 
$$\sum_{S'\in\Sigma'^k_A} \int_{S'} \sigma_{S'}^*\omega = \sum_{S\in\Sigma_{A}^k} \int_{S} \sigma_S^*\omega\ , $$ 
since if $\Sigma''_A$ is a different stratification 
then set
$\Sigma'_A$ to be a common refinement of $\Sigma_A$ and $\Sigma''_A$.
For any $k$-stratum $S\in\Sigma^k_A$ we can find $S'_1,...,S'_l\in \Sigma'^k_{A}$ such that $\bar S = \overline {\cup_j S'_j}$.
The set  $\bar S \setminus \cup_j S'_j$ is a semialgebraic set of dimension smaller than $k$ and therefore of Hausdorff $k$ dimensional measure $0$. Hence
$$\sum_{j=1}^l \int_{S'_j} \sigma_{S'_j}^*\omega =  \int_{S} \sigma_S^*\omega\ . $$

\end{proof}

\subsection*{Integration over Chains}
\begin{dfs}
Let $X\subset\R^n$. Define $C_k(X)$ to be the chain complex generated by semialgebraic continuous singular simplices, or simply singular simplices, $\sigma:\Delta^k\to X$, where $\Delta^k$ is the standard $k$ simplex in $\R^k$. Set $C^k(X):=\text{Hom}(C_k(X),\R)$ to be the complex of cochains with differential $d:=\partial^*$.

Let $c=\sum_{j=1}^L a_j\sigma_j \in C_k(X)$, where $\sigma_j$ is a singular simplex, $a_j\in\R$ for  $j=1,\dots,L$ and $\omega\in\Omega^k_\infty(X)$. Define the integral of $\omega$ over the chain $c$ by
$$ \int_c \omega := \sum_{j=1}^L a_j\int_{\sigma_j} \omega\  .$$
\end{dfs}

\subsection*{The De Rham Theorem}
\begin{dfs}
Let $X\subset\R^n$ . 
Define $H ^k (X)$ to be the cohomology of $C^k(X)$ and $H_\infty^k(X)$ to be the cohomology of $\Omega^k_\infty(X)$.
\end{dfs}

\begin{rk}
Semialgebraic homology was studied by many authors (see [K] for a list of 
references). In 1981 Hans Delfs [D] proved that semialgebraic homology is isomorphic to simplicial homology of a semialgebraic set over a real closed field. In 1996 Woerheide [Wo] showed that homology theory of singular definable simplices in a o-minimal structure satisfies Eilnberg-Steenrod axioms and therefore coincides with the standard singular homology theory.  
Complete proofs of comparison theorems for o-minimal homology (in particular for semialgebraic sets over $\R$ ) can be found in a recent paper by Edmundo and Wortheide [EWo].
\end{rk}

The main result of this article is
\begin{thm}\label{De Rham} (De Rham)
Let $X\subset\R^n$ be a compact set then
the map 
$$\psi : \Omega^k_\infty(X)\to  C^k(X),\ \ \  \psi(\omega)c := \int_c \omega $$ induces an isomorphism on cohomology.
\end{thm}

\section{Stokes' Theorem}\label{sec_stokes}
Stokes' formula for singular spaces was previously considered in the literature, see for example, [L] and [Paw]. In [L] Stokes' formula was considered for bounded subanalytic forms. In [Paw] Stokes' formula is proven for subanalytic leaves. In this section we give an alternative proof of Stokes' formula for semialgebraic chains. 

First we recall the definition of a cylindrical cell decomposition of $\R^n$ and prove some
basic but useful facts from semialgebraic geometry.

\begin{df} \label{ccd}
We define a {\bf cell of $\R^n$} by induction
on $n$. For $n=1$ a cell is either a point or an open interval.
For $n>1$, a cell in $\R^n$ is either a graph of a smooth function over a cell in $\R^{n-1}$ 
or a band, a set delimited by graphs of two smooth functions over a cell in $\R^{n-1}$, i.e. 
\mbox{$\{(x,y): x \in C', \xi_1(x) <y<\xi_2(x)\}$} where $C'$ is a cell in $\R^{n-1}$.
A finite collection $\C$ of disjoint cells of $\R^n$ is called {\bf cylindrical cell decomposition of $\R^n$} if the union of all cells in $\C$ covers $\R^n$.

A {\bf refinement} of a collection of cells $\C$ is a collection of cells $\C'$ such that 
every cell in $\C$ is a union of cells in $\C'$. We write $\C'\prec\C$.

We say that a collection $\C$ of cells of $\R^n$ is compatible with a set $A\subset\R^n$ if 
every cell that intersects $A$ is contained in $A$.

\end{df}

We will need the following definition for the next lemma.
\begin{df} Let $\A$ be a collection of cells in $\R^n$. We say that $\A$ satisfies the
{\bf frontier condition} if the boundary of each cell is a union of cells in $\A$.
\end{df}

\begin{lem}\label{str_front}
Let $A_1,\dots,A_l\subset\R^n$ be compact sets. There exists cylindrical cell decomposition of $\R^n$ 
compatible with $A_1,\dots,A_l$ that satisfies the frontier condition.
\end{lem}  
\begin{proof}
The proof is by induction on $n$. For $n=1$ the sets $A_1,\dots,A_l$ are points
and intervals so we can clearly find a decomposition of $\R$ into a union of points and open intervals that is compatible with the sets $A_i$.

Suppose that $n>1$. For the proof of this step we need to introduce a notation.
Suppose that $\A$ is a collection of cells in $\R^{n-1}$ and $\Xi$ is a collection of smooth and bounded functions $\xi_{C,j}:C\to\R$, $C\in\A$, $j=1,\dots,m_{C}$. Denote by $GB_{\Xi}(\A)$ the collection of cells in $\R^n$ that is obtained by taking graphs and bands of the 
functions in $\Xi$.
We will prove the following \\
{\bf Claim: } Let $\A$ be a collection of cells in $\R^{n-1}$ that satisfies the frontier 
condition and $\Xi$ be a collection of smooth functions $\xi_{C,j}:C\to\R$, $C\in\A$, $j=1,\dots,m_{C}$ such that $GB_{\Xi}(\A)$ is compatible with the sets $A_1,\dots,A_l$. 
Then, there exists a refinement $\tilde\A\prec\A$ and a collection of functions 
$\tilde\Xi$ of functions defined over the cells of $\tilde A$ such that 
$GB_{\tilde \Xi}(\tilde\A)$ refines $GB_{\Xi}(\A)$, satisfies the frontier condition
and is compatible with the sets $A_1,\dots,A_l$.

Before we prove the claim let us show that it implies the step of the induction.
Let $\pi_n:\R^n\to\R^{n-1}$ be the projection to the first $n-1$ coordinates.
Let $\C$ be a cylindrical cell decomposition of $\R^n$ compatible with $A_1,\dots,A_l$.
Set $\C':=\pi_n(\C)$. By induction hypothesis there exists a refining cell decomposition $\A\prec\C'$ that satisfies the frontier condition. Let $\Xi'$ be the collection of
all functions defined over the cells of $\C'$ such that $GB_{\Xi'}(\C')=\C$.
Set $\Xi$ to be the collection of functions that is obtained by restricting the functions in $\Xi'$ to the cells of $\A$. It follows from the claim that there exists 
a refinement $\tilde \A\prec \A$ and a collection of functions $\tilde \Xi$ such that 
$\B:=GB_{\tilde \Xi}(\tilde \A)\prec \C$ satisfies the frontier condition and is compatible with 
$A_1,\dots,A_l$. Therefore, $\B$ is the desired cell decomposition.

To complete the lemma we prove the claim. The proof of the claim is by induction on 
$k=\max \dim \{C: C\in \A\}$.
If $k=0$ then $\A$ is a finite collection of points and therefore the functions in $\Xi$
are constant functions. Set $\tilde \A:=\A$ and $\tilde \Xi:=\Xi$.

Suppose that $k>0$. Set 
$$ \A_1 = \A-\{\text{cells of $\A$ of dimension $k$}\}$$
Let $\hat\A_1$ be a cell decomposition of $\pi_n^{-1}(|\A_1|)$
that is compatible with $\pa \Gamma_{\xi_C,j}$, $C\in \A$, $\xi_{C,j}\in \Xi$ ,$j=1,\dots,m_{C}$.
Let $\hat \A_1'$ be the cells of $\hat \A_1$ that are contained in $\R^{n-1}$.
Note that $\hat \A_1'\prec \A_1$. 
We may assume that $\hat \A_1'$ satisfies the
frontier condition, since if not, by induction hypothesis of the lemma, we may
find a refinement of $\hat \A_1'$ that satisfies the frontier condition.
Set $\Xi'_1$ to be a collection of functions defined over the cells of $\hat\A_1'$ 
such that $GB_{\Xi'_1}(\hat\A_1')=\hat \A_1$.
Since $\max \dim \{C\in\hat\A_1'\} < k$ we may apply the induction hypothesis of the claim
to $\hat\A_1'$ and $\Xi_1'$ to obtain a refinement $\B\prec \hat\A_1'$ and a collection of functions $\Xi_{\B}$ such that $GB_{\Xi_{\B}}(\B)$ satisfies the frontier condition and
compatible with $A_1,\dots,A_l$. Now, let $\C$ be the collection of cells in $\A$ of 
dimension $k$ and set $\Xi_{\C}$ to be the subset of $\Xi$ consisting of all the functions
that are defined over the cells in $\C$.
Set $\tilde \A := \B\cup\C$ and $\tilde \Xi:=\Xi_{\C}\cup\Xi_{\B}$.

Let us check that $\mathcal D:=GB_{\tilde \Xi}(\tilde \A)$ satisfies the frontier condition.
First suppose that $D$ is a graph or band over a cell $D'\in \B$.
In this case $D\in GB_{\Xi_{\B}}(\B)$ so it satisfies the frontier
condition by construction.
Now suppose that $D$ is a graph or a band over a cell $D'\in \C$.
\\
{\bf case I: $D=\Gamma_{\xi_{D'}}$.} In this case $\xi_{D'}\in\Xi_{\C}$ and $\pa D=\pa \Gamma_{\xi_{D'}}$.
Recall that $GB_{\Xi_{\B}}{(\B)}$ is compatible with $\pa \Gamma_{\xi_{D'}}$. 
In particular, $\pa D$ is contained in a union of cells of $GB_{\Xi_{\B}}(\B)$.
\\
{\bf case II:} $D$ is a band delimited by graphs of two functions $\xi_1 < \xi_2$. 
In this case $\xi_1,\xi_2 \in\Xi_{\C}$ and 
$$ \pa D=\Gamma_{\xi_1}(D')\cup\Gamma_{\xi_1}(D')\cup R, $$
where $R$ is a set that projects into $\pa D'$. Now, since cells in $\R^{n-1}$ satisfy the 
frontier condition, there exist cells $D_1,\dots,D_p$ such that $\pa D'=\cup_{j=1}^{p} D_j$.
As in case I,  $GB_{\Xi_{\B}}(\B)$ is compatible with $\pa \Gamma_{\xi_i}$ for $i=1,2$ and therefore, the set $R$ is given by graphs and bands of some functions in $\Xi_B$ that are defined 
over the cells $D_1,\dots,D_p$.
\end{proof}

\begin{lem}\label{str_to_str}
Let $X,Y\subset\R^n$ be compact subsets with stratifications $\Sigma_X$ and $\Sigma_Y$, $f:X\to Y$ a map. There exist refining stratifications $\Sigma'_X$ and $\Sigma'_Y$, that 
are cell decompositions, such that 
$f(S)\in\Sigma'_Y$ for any $S\in \Sigma'_X$ and moreover, if $\dim f(S)=\dim S$ then $f|_S$ is a diffeomorphism.
\end{lem}  
\begin{proof}
Let $Gr:=\{(f(x),x):x\in\ X\}\subset Y\times X$. By Lemma \ref{str_front}, there exists cylindrical cell decomposition $\C$ of $\R^{2n}$ that is compatible with $Gr$
and $Y\times X$, refines $\Sigma_Y\times\Sigma_X$ and satisfies the frontier condition. 
Define $\Sigma'_X$ to be the union of projections of all the the cells in $Gr$ to $X$ and
similarly, $\Sigma'_Y$ to be the union of projections of all the cells in $Y\times X$ to $Y$.

Note that $\Sigma'_Y$ is a stratification since by 
definition of cell decompositions cells in $Y\times X$ project to cells in $Y$.
Clearly $f(x)=\pi_Y\circ i_{Gr}(x)$ where $i_{Gr}(x)=(f(x),x)$ and therefore $f$ maps cells
to cells.

Suppose that $\dim f(S)=\dim S$ for $S\in\Sigma'_X$. 
Suppose that the cell $C_l:=\{(f(x),x):x \in S\}\subset Gr$ was obtained inductively as follows.
The first cell $C_1$ is a cell in $Y$ and $C_j$ is a graph or a band over the cell $C_{j-1}$.
Clearly,  
$$ f(S)=\pi_Y(\{(f(x),x):x \in S\}) = C_1 .$$ 
It follows that $\dim C_1=\dim C_l$ and therefore each $C_j$, $j=1,\dots,l$ is a graph.
In particular $\pi_Y$ restricted to $C_l$ is a one to one map 
and therefore $f|_S$ is a diffeomorphism.
\end{proof}

The following Lemma is a version of the well known Wing Lemma (cf. [BCR] Theorem 9.7.10, [Paw] Theorem 1.1 and [Wa] Proposition on page 342). 
\begin{lem}\label{wing}
Let $X\subset \R^n$ be a locally closed manifold of dimension $k$. 
There exists a set $B\subset \overline X$, $\dim B<k-2$ such that 
every point $p\in \pa X-B$ has a neighborhood $W\subset \overline X$ that is a union of finitely many manifolds $W_1,\dots,W_m$ with a common boundary $W\cap \pa X$.
Moreover, if $X$ is a cell then $m=1$.
\end{lem}

\begin{proof}

Let $\C$ be a cell decomposition of $\R^n$ that is compatible with $\overline X$.
Let $S\in \C$, $S\subset X$, $\dim S = k$ and $\overline S \cap \pa X\neq \emptyset$.
We claim that there exists a set $B\subset \pa X$, $\dim B<k-2$ such that 
every point $p\in \pa S-B$ has a neighborhood $U\subset \overline S$ that
is a manifold with boundary $U\cap\pa S$.
Before we prove this claim let us show that the lemma follows from it.
Indeed, for each point $p\in \pa X-B$ there exist a cell $S'$ of dimension $(k-1)$ and
cells $S_1,\dots,S_m$ of dimension $k$ such that $S'\leq S_j$ for all $j=1,\dots,m$ and
$\cup_{j=1}^m \overline S_j$ contains an open neighborhood of $p$ in $\overline X$. Take $W_j\subset \overline S_j$ to be an open set such that $W_j$ is a manifold with boundary 
$W_j\cap\pa S_j$ for $j=1,\dots,m$. Set $W:=\cup_{j=1}^m W_j$. 

Next, we prove the claim.
Let $G:S\to \G_{n}^{k}$ be the Gauss map, $x\mapsto T_x S$. 
Let $\Gamma:=\Gamma_G(S)$ be the graph of $G$.
Since $S$ is semialgebraic, $\dim \pa \Gamma< \dim \Gamma=k$. Therefore, the set $E:=\pa \Gamma \cap (\pa S \times \G_n^{k})$ is of dimension at most $k-1$. Let $\pi :\R^n\times\G_n^{k}\to \R^n $ be the standard projection. Note that $\pi(E)\subset \pa S$ and hence we have $\pi|_E:E \to \pa S$.
Thus we conclude that there exists a set $B'\subset \pa S$ such that any fiber of $\pi$ over 
$\pa S - B'$ must be a 
connected set of dimension $0$ and therefore must be a point.
It follows that the Gauss map can be uniquely extended to a continuous map in a neighborhood of any point $p\in \pa S - B'$. We will keep denoting the extension of the Gauss map by $G$. 
Let $G':\pa S - B'\to \G^{k}_n$, $G'(p):=G|_{\pa S - B'}=\lim_{q\to p}G(q)$. The map $G'$ is a continuous semialgebraic map and hence $\pa S - B'$ can be stratified so that $G'$ is smooth on every stratum.
Let $B''\subset\pa S - B'$ be the union of all the strata of dimension smaller than $k-1$.
Clearly, $G'$ is smooth on $\pa S - (B'\cup B'')$.

Let $\Sigma$ be a Whitney (a) stratification of $\overline S$ and set $B:=B'\cup B''\cup |\Sigma^{(k-2)}\cap\pa S|$.
Let $p\in\pa S - B$. In particular $p$ belongs to some stratum $S'\in\Sigma$ of dimension $(k-1)$. 
We will show that $\overline S$ is a manifold with boundary near $p$.
Set $L_q:=T_q S'$. 
By Whitney (a) condition we have $L_q \subset G(q)$.
Since $G$ is continuous we may pick a small enough ball $U_0:=B(p,\eps)\cap\overline S$ such that the Grassmanian distance between $G(q)$ and $G(p)$, $q\in U_0$, is so small that there exists an orthogonal projection $\pi_{G(p)}:\R^n \to G(p)$ such that $\pi_{G(p)}|_{G(q) }$ is one to one for every $q\in U_0$.
Similarly, by possibly shrinking $U_0$, we may assume that there exists 
an orthogonal projection $\pi_{L_p}:\R^n \to L_p$ such that $\pi_{L_p}|_{L_q }$ is one to one for every $q\in U'_0:=U_0\cap S'$. 
In other words, the neighborhood $U_0$ is chosen to be so small that the tangent spaces 
at points of $U_0$ are nearly parallel to $G(p)$ and the tangent spaces to $U_0\cap S'$ are
nearly parallel to $L_p$.

Let $n_q\in L{_q}^{\perp}\subset G(q)$ be a unit normal vector pointing inside $S$.
Note that since the spaces $G'(q)$ and $L_q$ vary smoothly for $q\in U'_0$ it follows that
$q\mapsto n_q$ is a smooth map.
Now we construct a diffeomorphism from $\gamma:U'_0\times[0,\eps)\to \overline S$ for some small $\eps>0$. 
For each point $q\in U'_0$ consider 
an affine space $\tilde L_q$ centered at $q$ and parallel to $L_q$.
Let $\pi_{\tilde L_q}:\R^n \to \tilde L_q$  be an orthogonal projection.
The fiber $\pi_{\tilde L_q}^{-1}(q)\cap S$ is an arc that can be arc-length parametrized by $\gamma(q,t)$ with $\gamma(q,0)=q$. Set $\gamma_{q,s}$ to be the set $\{\gamma(q,t):0\leq t\leq s\}$. Choose $\eps$ to be small enough so that $\gamma_{q,\eps}$ are disjoint for $q\in U'_0$.
To see that $\gamma$ is smooth note that $\gamma$ has expansion
$$ \gamma(q,t) = q+n_qt + \dots $$
\end{proof}

\begin{thm}\label{StokesThm}(Stokes' formula)
Let $X\subset\R^n$ be a compact set, $\sigma$ be a singular $k$-simplex and $\omega$ be an $L^\infty$ $(k-1)$-form on $X$. Then
\begin{equation}\label{stks}
 \int_{\sigma}d\omega=\int_{\pa \sigma}\omega\ . 
\end{equation}
\end{thm} 
\begin{proof}

Suppose that $(\omega,\Sigma'_X)$ is a stratified form.
Let $\Sigma_\Delta$  and $\Sigma_X\prec\Sigma'_X$ be stratifications of $\Delta$ and $X$ given by Lemma \ref{str_to_str} for the map $\sigma$.
It is enough to prove that   
$$ \int_{S}d\omega=\int_{\pa S}\omega \ ,$$
where $S=\sigma(S')\in\Sigma_X$, with $S'\in\Sigma_\Delta$, $\dim S' = k$, such that $\sigma|_{S'}$ a diffeomorphism onto $S$. Note that $S'$ is an orientable manifold since it is a cell.
Fix such $S$ and
set 
$$ S_\eps:=\{x\in S: d(x,\overline {S} - S)\geq\eps \}\ ,$$
where $d(.,.)$ denotes the Euclidean metric.
Note that $S_\eps$ is a smooth manifold with boundary and $\omega$ is a smooth form on it, so 
by the classical Stokes' formula 
$$\int_{S_\eps}d\omega = \int_{\pa S_\eps}\omega\ . $$  
Therefore, to prove (\ref{stks}), we only have to show that 
\begin{equation}\label{eq_1}
\lim_{\eps\to 0} \int_{S_\eps}d\omega = \int_{S}d\omega 
\end{equation}
and
\begin{equation}\label{eq_2}
\lim_{\eps\to 0} \int_{\pa S_\eps}\omega=\int_{\pa S}\omega \ .
\end{equation}
Equality (\ref{eq_1}) is clear since $d\omega$ is integrable on $S$ (because $d\omega$ is bounded) and 
$\mu_k(S -S_\eps)\to 0$ as $\eps\to 0$, where $\mu_k(\cdot)$ denotes the $k$-dimensional Hausdorff measure.
To prove equality (\ref{eq_2}), let $\delta\in\R_+$ and suppose that $B\subset\pa S$,  $\dim B < k-1$, is given by Lemma \ref{wing}. Set 
$$ A_\delta:=\{x\in\pa S : d(x,\Sigma_X^{(k-2)}\cup B)\geq\delta \}\ , $$
and
$$ B_\delta:=\{x\in S : d(x,B\cup\Sigma_X^{(k-2)}\cap \pa S)\leq\delta \}\ .  $$ 
Note that $A_\frac{\delta}{2}$ is a smooth manifold with boundary and for every interior (in $A_\frac{\delta}{2}$) point $p\in A_\frac{\delta}{2}$ there exist
open neighborhoods $\tilde U_p\subset \overline S$ and  $U_p\subset A_\frac{\delta}{2}$ such 
the set $\tilde U_p$ is manifold with boundary $U_p$.
Let $\phi^p: D\times[0,\eps_p)\to \tilde U_p$, $D\subset\R^{k-1}$ an open subset, $\eps_p>0$, be a smooth parametrization of $\tilde U_p$ 
such that $\phi^p|_{D\times\{0\}}$ is a parametrization of $U_p$. 
To simplify the notations we write $\phi(.,.)$ instead of $\phi^p(.,.)$ when $p$ is clear from the context.
We may assume, by possibly shrinking $U_p$, that $D$ is a unit ball.
We may also assume that the curves $t\mapsto \phi(x,t)$ with $x\in D$ fixed, are arc length parametrized. Since otherwise, we may consider a change of coordinates $(x,t)\mapsto(x,t')$ given by
\begin{equation*}
t:=\int_0^{t'} |\frac{\pa}{\pa t}\phi(x,s)|ds .
\end{equation*}

Observe that since the curves $t\mapsto \phi(x,t)$ are arc length parametrized we have 
$\pa S_t \cap \tilde U_p = \{\phi(x,t): q\in D\}$. In particular it follows that
for all $t\in[0,\eps_p)$ the map $\phi_t:D\to \pa S_t \cap \tilde U_p$, $\phi_t(x):=\phi(x,t)$ is a diffeomorphism. Clearly, 
$\phi_t$ tends uniformly to $\phi|_{D\times\{0\}}$ in $C^1$ topology as $t\to 0$.\\
Therefore, for each $\tilde U_p$  we have
\begin{equation}\label {int_U_p}
\lim_{\eps\to 0} \int_{\pa S_\eps\cap \tilde U_p}\omega=
\lim_{\eps\to 0} \int_{\phi_\eps(D)}\omega=
\lim_{\eps\to 0} \int_{D}\phi_\eps^*\omega=
\int_{\pa S\cap\tilde U_p}\omega\ . 
\end{equation}
Since $S$ is bounded, the set $A_\delta\subset A_{\frac{\delta}{2}}$ is compact and therefore we may choose finitely many sets $U_i:=U_{p_i}$ that cover $A_\delta$. 
Set $\tilde U'_i:=\phi^{p_i}((U_i\cap A_\delta)\times[0,\eps_0))$ where $\eps_0=\min_{i} \eps_{p_i}$ and $U_\delta:=\cup \tilde U'_i$.
Note that if $\delta>0$ is fixed then for $\eps>0$ small enough we have $\pa S_\eps \subset U_{\delta} \cup B_\delta$ and therefore
$$ \int_{\pa S_\eps} \omega = \int_{\pa S_\eps \cap U_\delta} \omega + \int_{\pa s_\eps\cap (B_\delta-U_\delta)}\omega\  .$$
Let $\{\varphi_i\}$ be a partition of unity subordinate to the cover $\{\tilde U'_i\}$. 
Thus by (\ref{int_U_p}) we obtain
$$ \lim_{\eps\to 0}\int_{\pa S_\eps \cap U_\delta} \omega = \lim_{\eps\to 0}\sum \int_{\pa S_\eps \cap \tilde U'_i} \varphi_i\omega = \sum \int_{\pa S \cap \tilde U'_i\cap A_\delta} \varphi_i\omega
= \int_{A_\delta} \omega\ . $$


To complete the proof of (\ref{eq_2}), it is enough to show that $\mu_{k-1}(B_\delta\cap\pa S_\eps )$
is small in terms of $\delta$ i.e., bounded by a function $g(\delta)$ 
with $g(\delta)\to 0$ as $\delta \to 0$ for every $\eps>0$.
Since then, because $\omega$ is integrable (bounded), 
$$\left| \int_{\pa S_\eps\cap (B_\delta-U_\delta)} \omega\right|
\to 0
\text{ as } \delta \to 0\ \text{ uniformly in }\eps  .$$ 
For that matter, we use the well known Cauchy-Crofton formula [F] that can be formulated as follows.
Let $A\subset\R^n$ and set 
$$ K_j^P(A):=\{q\in P : \#\left(\pi_P^{-1}(q)\cap A \right)=j \}, $$
where $\pi_P$ is the orthogonal projection from $\R^n$ to $P$.
The $k$-dimensional Hausdorff measure of $A$ is given by
\begin{equation}\label{CC_statement}
\mu_{k}(A)=
\int_{P\in\G^{k}_n}\sum_{j=1}^{\infty}j\mu_{k}(K_j^P(A)) d\gamma(P)\ ,
\end{equation}
where 
$\gamma$ is a finite measure on $\G^{k}_n$.
Substituting $B_\delta\cap\pa S_\eps $ in formula (\ref{CC_statement}) we obtain

\begin{equation}\label{Cauchy-Crofton}
\mu_{k-1}(B_\delta\cap\pa S_\eps )=
\int_{P\in\G^{k-1}_n}\sum_{j=1}^{N(\delta,\eps)}j\mu_{k-1}(K_j^P(B_\delta\cap\pa S_\eps)) d\gamma(P)\ ,
\end{equation}
where 
$$N(\delta,\eps):=\max\{\#\left(\pi_P^{-1}(q)\cap B_\delta\cap\pa S_\eps\right)<\infty:P\in \G^{k-1}_n, q\in P\}\ .$$ 
We remark that the number $N(\eps,\delta)$ can be bounded from above by an integer $N_0$ independent of $\eps$ and $\delta$, as a consequence of uniform boundness principle for families of semialgebraic sets (see [BCR]). 
Since $\gamma$ is a finite measure, it is enough to bound each 
$\mu_{k-1}(K_j^P(B_\delta\cap\pa S_\eps))$ in terms of $\delta$ where $P\in\G^{k-1}_n$.
Note that 
$$K_j^P(B_\delta\cap\pa S_\eps)\subset \pi_P(B_\delta)\ .$$
Denote by $Y$ the set $B\cup \Sigma^{(k-2)}_X\cap \pa S$.
Clearly $\dim Y\leq k-2$ and 
$B_\delta=\{x\in S: d(x,Y)\leq \delta \}$. Therefore
$$\pi_P(B_\delta)\subset \{x\in \pi_P(S):d(x,\pi_P(Y))\leq \delta \} \ .$$
Since $\dim \pi_P(Y) < k-1$ it follows that $\mu_{k-1}(\pi_P(B_\delta))\to 0$ as $\delta \to 0$.

Set 
 $$ g(P,\delta):=\mu_{k-1}(\pi_P(B_\delta))\ .$$
Clearly the function $g$ is bounded and therefore by Lebesgue dominated convergence theorem 

$$ \int_{\G^{k-1}_n} g(\delta,P)d\gamma(P)\to 0\text { as } \delta\to 0\ .$$
From (\ref{Cauchy-Crofton}) it follows that:
\begin{equation}
\mu_{k-1}(B_\delta\cap\pa S_\eps )\leq
\int_{P\in\G^{k-1}_n}\sum_{j=1}^{N_0}jg(P,\delta) d\gamma(P)
\text{ as } \delta\to 0\ . 
\end{equation}

\end{proof}

\section{De Rham Theorem}\label{sec_de_rham}
\subsection{Elementary forms}\label{elm_form}$\ $\\
In this section we recall the concept of elementary differential forms introduced in [W Ch. IV, 27]
and adapt it to our setting.

Let $X\subset\R^n$ be a compact set, and $T:|K|\to X$ be any triangulation, that is, 
a homeomorphism from a simplicial complex $K$, $|K|\subset \R^n$, that is smooth on every open simplex $\sigma\in K$.


Let $S_j(K)$ be the real vector space generated by all $j$-simplices in $K$ and let
$\=:S_j(K)\to S_{j-1}(K)$ be the standard boundary operator defined in the usual way.
Endow $S^j(K):=Hom(S_j(K),\R)$ with the coboundary operator $d:=\pa^*$. 
Identify the canonical basis of $S_j(K)$ consisting of 
all $j$-simplices $\sigma^j_l$ of $K$, $1\leq l\leq \dim S_j(K)$
with a basis of $S^j(K)$  by 
$\sigma^j_l \leftrightarrow <\sigma^j_l,\cdot>$ where $<.,.>$ denotes the inner product 
defined by $<\sigma^j_l,\sigma^j_i>=\delta_{l,i}$.

\begin{df}
Given $\sigma\in K$, {\bf the star of $\sigma$}, $St(\sigma)$, 
is defined as the union of all open simplices that contain $\sigma$ in their 
closure.
\end{df}

Let $T:|K|\to X$ be a triangulation and let 
$\{Q_i\}_{i\in \I}$ be a finite cover of $X$ with $Q_i$ open in $\R^n$ for $i\in\I$, such that $X\cap Q_i\subset St(T(p_i))$ where $p_i$ is a vertex.
Let $\phi_i$ be a smooth partition of unity subordinate to the cover $\{Q_i\}_{i\in \I}$.
Any $j$-simplex $\sigma\in K$ can be represented by its vertices as $(p_{i_0},\dots, p_{i_j})$. 
Set
$$ \phi_{T,\sigma}:=j!\sum_{k=0}^{j}(-1)^{k}\phi_{i_k}d\phi_{i_0}
                  \wedge\dots\hat{d\phi_{i_k}}\dots\wedge d\phi_{i_j}\ . $$
The forms $\phi_{T,f}$ are called the {\bf elementary $j$-forms}. We extend this definition to all $S^j(K)$ by linearity via the identification of $S_j(K)$ with $S^j(K)$, that is 
if $f\in S^j(K)$, $f=\sum_k a_k\sigma^j_k$ then 
$$ \phi_{T,f}:=\sum_k a_k\phi_{T,\sigma^j_k}\ . $$

\begin{prop}\label{elem_prop}
For any triangulation $T:|K|\to X$  the elementary forms have the following
properties:

\begin{enumerate}
\item $\phi_{T,df}|_X=d\phi_{T,f}|_X$ for all $f\in S^j(K)$,
\item $\psi\phi_{T,f}=f$ for all $f\in S^j(K)$.
\end{enumerate} 
\end{prop}
For a proof of this proposition see [W Ch. IV, 27].

\begin{rk}
For any  triangulation $T$ of $X$, the elementary forms are $L^\infty$ forms since they are  restrictions of smooth forms from the ambient $\R^n$.
\end{rk}

Let $\Omega^j_{elm}(T,X)$ be the linear span of the set of all elementary $j$-forms and exterior derivatives of $(j-1)$-forms. By construction, $(\Omega^\bullet_{elm},d)$ is a chain complex.
Denote by $H_{elm}^k(T,X)$ the $k^{th}$ cohomology group of $(\Omega^\bullet_{elm},d)$.
As a consequence of Proposition \ref{elem_prop} we have
\begin{prop}\label{cohom_elem_isom}
There exists an isomorphism  $\Phi_T:H^k(X)\to H^k_{elm}(T,X)$.
\end{prop}
\begin{proof}
Define $\Phi_T$ as follows.
Let $f$ be a closed singular $k$-cochain. Then $[f]\in Hom(H_k(X),\R)$, where $[f]$ denotes an equivalence class of $f$ in $Hom(H_k(X),\R)$. 
Since simplicial cohomology is isomorphic to singular cohomology, there exists a closed simplicial $k$-cochain $f'\in S^k(K)$ that has the same class as $f$ in singular $k$-cohomology of $X$. Note that $\phi_{T,f'}$ is a closed form by Proposition \ref{elem_prop} (1).
Set $$ \Phi_T([f]):= [\phi_{T,f'}]. $$

{\it The map $\Phi_T$ is well defined.} Indeed, if $f''$ is another element of $S^k(K)$ 
that defines the same cohomology class as $f'$ then, 
$f''-f'=dg$ for some simplicial $(k-1)$-cochain $g$, so 
$\phi_{T,f''}-\phi_{T,f'}=\phi_{T,dg}=d\phi_{T,g}$. 

{\it The map $\Phi_T$ is one to one.} Suppose that $\Phi_T([f])=[\phi_{T,f'}]$, $f-f'=dg$, $g\in C^{k-1}(X)$ and $\phi_{T,f'}=d\phi_{T,f'_1}$, $f'_1\in S^{k-1}$. It follows from Proposition \ref{elem_prop} (1) that $\phi_{T,f'}=\phi_{T,df'_1}$ and by Proposition \ref{elem_prop} (2)
we have 
$$ f'=\psi\phi_{T,f'}=\psi\phi_{T,df'_1}=df'_1. $$

{\it The map $\Phi_T$ is onto.} Let $\phi_{T,f}$ be a closed elementary $k$-form.
By Proposition \ref{elem_prop} (1) we have $0=d\phi_{T,f}=\phi_{T,df}$ and 
then, by Proposition \ref{elem_prop} (2), we get $0=\psi\phi_{T,df}=df$.
Now it is clear that $\Phi_T([f])=\phi_{T,f}$.

\end{proof}

\subsection {Proof of The De Rham Theorem}
The proof of the De Rham Theorem relies on the following Poincar\'{e} Lemma for $L^\infty$ forms. 

\begin{thm}\label{Poincare}(Poincar\'{e} Lemma) Let $\omega$ be a closed smooth $L^\infty$ $k$-form on $X\subset\R^n$ and $p\in X$. There exists a neighborhood $U_p$ of $p$ in $X$, and a smooth $L^\infty$ $(k-1)$-form $\gamma$ defined on $U_p$ such that $\omega= d\gamma$ in $U_p$.
\end{thm} 
We prove this theorem is in Section \ref{smoothing}.

\begin{rk}\label{rk_poincare}
In the case of a smooth form on a smooth manifold the Poincar\'e lemma can be proved for
any star shaped domain independently of the form in question. In Theorem \ref{Poincare} the neighborhood for which the theorem is formulated depends on the given form. Nevertheless, for a given closed $L^\infty$ $k$-form $\omega$ on a compact set $X\subset\R^n$, 
there exists $\eps>0$ such that for any $p\in X$ there exists an $L^\infty$ $(k-1)$-form  $\eta_p$ 
in $B(p,\eps)\cap X$ such that $\omega=d\eta_p$ there. Indeed, let $\{B(p,\eps_p)\cap X\}_{p\in X}$ be a cover of $X$ such that $\omega=d\eta_p$ on $B(p,\eps_p)\cap X$. Since $X$ is compact, we can 
pick a finite sub-cover $\{U_{p_i}\}_{i=1}^{L}$ and set $\eps$ to be the Lebesgue number for this cover.
\end{rk}
For the next theorem we will need a simple lemma from linear algebra.
\begin{lem}\label{lin_alg}
Let $V,W_1$ and $W_2$ be real finite dimensional vector spaces,
\newline 
\mbox{$ \varphi_j:V\twoheadrightarrow W_j,\  j=1,2\text{ be surjective homomorphisms}$} 
and $f\in Hom(V,\R)$ such that 
$$ker\  f \supset ker\ \varphi_1\ \cap\  ker\ \varphi_2 .$$ 
There exist $g_1\in Hom(W_1,\R)$ and $g_2\in Hom(W_2,\R)$ such that 
$$ f=\varphi_1^*g_1+\varphi_2^*g_2 . $$
\end{lem}
\begin{proof}
Let $\psi:V\to W_1\oplus W_2$, $\psi(x)=(\varphi_1(x),\varphi_2(x))$. 
Clearly $\psi$ yields an isomorphism 
$$V/(ker\ \psi)\stackrel{\sim}{\rightarrow} Im(\psi), $$
let $\xi:Im(\psi)\to V/(ker\ \psi)$ be its inverse.
Note that 
$$ ker\ \psi=ker\ \varphi_1\ \cap\ ker\ \varphi_2\ \subset ker\ f,$$ 
and therefore the functional $f$ defines
an element of $Hom(V/(ker\ \psi),\R)$ that we continue denoting by $f$.
Let $g':=\xi^*f\in Hom(Im(\psi),\R)$ and let $g$ be any extension of $g'$ to $W_1\oplus W_2$. Set 
$$ g_1(x):=g(x,0) \text{ and } g_2(x):=g(0,x). $$
We claim that $f=\varphi_1^*g_1+\varphi_2^*g_2$. Indeed, 
\begin{eqnarray*}
 \varphi_1^*g_1(x) + \varphi_2^*g_2(x) &=& g_1(\varphi_1(x))+g_2(\varphi_2(x)) \\ 
&=& 
 g(\varphi_1(x),0)+g(0,\varphi_2(x))\\ &=& 
 g(\varphi_1(x),\varphi_2(x)) \\ &=&
 g'(\varphi_1(x),\varphi_2(x))\\ &=& 
 f(\xi(\varphi_1(x),\varphi_2(x))) \\&=& 
 f(\xi(\psi(x)))\\&=&
 f(x).
 \end{eqnarray*} 

\end{proof}

\begin{thm}\label{inj}(Injectivity of the De Rham isomorphism) 
Let $X$ be a compact set and $\omega$ any smooth closed $L^\infty$ $k$-form defined on $X$ such that $\int_c \omega=0$ for any $c\in H_k(X)$, then there exists a smooth $L^\infty$ 
$(k-1)$-form $\eta$ on $X$ such that $\omega=d\eta$.
\end{thm} 
\begin{proof}
We prove the theorem by induction on $k$. For $k=0$ the theorem is trivial. Let $k>0$ 
and let $\omega$ be a closed $L^\infty$ $k$-form on $X$.
\newline 
\textbf{STEP 1:}
Suppose that $X=A\cup B$, $\omega=d\eta_A$ in $A$ and $\omega=d\eta_B$ in $B$ where $A$ and $B$ are closed sets. Suppose also that $X=A'\cup B'$ where $A'\subset A$ and $B'\subset B$ are closed sets. 
We will prove in this step that there exists a smooth $L^\infty$ form $\eta$ such that $\omega=d\eta$ in $X=A\cup B$. 
Set $\eta_{AB}:=\eta_A-\eta_B$ in $A\cap B$, and note that it is a closed $(k-1)$-form there. 
We claim that there exist closed forms $\phi_A$ near $A$ and $\phi_B$ near $B$ such
that 
$$ \int_{[c]} (\eta_{AB}-\phi_A-\phi_B) =0 \text{ for any }[c]\in H_{k-1}(A\cap B)\ . $$ 
Let us assume for a moment that we have found such $\phi_A$ and $\phi_B$. Then, 
by induction hypothesis there exists a $(k-2)$-form $\xi'_{AB}$ on $A\cap B$ such that $$\eta_{AB}-\phi_A-\phi_B=d\xi'_{AB}. \ $$ 
Let $\varphi$ be a smooth function on $\R^n$ that is 
identically $1$ near $A'\cap B'$ and 
$0$ near  $X-A\cap B$. Let 
$\xi_{AB}=\varphi\xi'_{AB}$ and set 

\[
\eta = \left\{ \begin{array}{ll}
         \eta_A-\phi_A & \text{ on } A'  \\
        \eta_B+d\xi_{AB}+\phi_B & \text{ on }  B'\ .
         \end{array} \right.
\]
Clearly $\eta$ is a well defined form and $d\eta=\omega$. 

Now we only have to find the forms $\phi_A$ and $\phi_B$. 
In order to apply Lemma \ref{lin_alg} we set 
$$V:=H_{k-1}(A\cap B)\ ,\  W_1:=Im(H_{k-1}(A\cap B)\to H_{k-1}(A))\ , $$ 
$$ W_2:=Im(H_{k-1}(A\cap B)\to H_{k-1}(B))\ \text{ and } f([c])=\int_{[c]}\eta_{AB}\ ,$$ 
with $\varphi_1$ and $\varphi_2$ being the maps induced by inclusions of $A\cap B$ into $A$ and $B$. 
Note that 
$$ ker\ f\ \supset ker\ \varphi_1\ \cap\  ker\ \varphi_2$$
since if 
$\varphi_1([c])=\varphi_2([c])=0$ for $[c]\in H_{k-1}(A\cap B)$ 
then there exist chains $c_1$ in $A$ and $-c_2$ in $B$ such that
$c=\pa c_1=-\pa c_2$, but then 
$$f([c])=\int_{[c]}\eta_A-\eta_B=\int_{\pa c_1}\eta_A + \int_{\pa c_2}\eta_B
=\int_{c_1+c_2}\omega=0\ . $$  
Therefore, by Lemma \ref{lin_alg} there exist $g_1\in Hom(W_1,\R)$ and $g_2\in Hom(W_2,\R)$ such that 
$f=\varphi^*g_1+\varphi^*g_2$. Since we work over $\R$, we may identify $H^{k}(Z)$ with 
$Hom(H_k(Z),\R)$ for any space $Z$. Let $g'_1$ and $g'_2$ be any extensions of $g_1$ and $g_2$ to $H^{k-1}(A)$ and $H^{k-1}(B)$. Set $\phi_A$ and $\phi_B$ to be the elementary forms 
corresponding to $g_1'$ and $g_2'$ (see Sections \ref{elm_form} and Proposition \ref{cohom_elem_isom}). Since $g'_1$ and $g'_2$ are closed cochains, the forms
$\phi_A$ and $\phi_B$ are closed. Since $\int_{[c]} \phi_A=g'_1([c])$ for $[c]\in H_{k-1}(A)$ and
$\int_{[c]} \phi_B=g'_2([c])$ for $[c]\in H_{k-1}(B)$ it follows that for any $[c]\in H_{k-1}(A\cap B)$
$$ \int_{[c]}(\eta_{AB}-\phi_A-\phi_B)=f([c])-g'_1([c])-g'_2([c])= 0\ . $$
And this  concludes the first step.
\newline
\textbf{STEP 2:}
Let $\U=\{U_i\}_{i=1}^N$ be a cover of $X$ such that
$\omega=d\eta_i$ near each $U_i$ which is possible due to compactness of $X$ 
and Theorem \ref{Poincare} (see Remark \ref{rk_poincare}). Set $V_l=\bigcup_{i=1}^{l}U_i$.
We claim that for each $l\leq N$ there exists a form $\xi_l$ 
near $V_l$ such that $\omega=d\xi_l$ there.
We prove this claim by induction on $l$. 
For $l=1$ the claim follows by the choice of the cover. 
Let $l>1$.  
By STEP 1, applied to $V_l$ and $U_{l+1}$ there exists 
a form $\xi_{l+1}$ near $V_{l+1}$ such that $\omega=d\xi_{l+1}$.
Therefore the claim is proven and the theorem follows for $l=N$.
\end{proof}

\textit{Proof of Theorem \ref{De Rham}.\ } (Compare with [W] Ch. IV ,27).
Theorem \ref{inj} implies that the map induced by $\psi$ on cohomology is injective. 
To see that the latter map induces a surjective map on cohomology, let $f$ be a closed 
singular $k$-cochain. Let $T$ be any triangulation of $X$ and choose $f'$ to be a simplicial cochain
in the same cohomology class as $f$. 
By Proposition \ref{cohom_elem_isom}, all the forms in $\Phi_T([f])$ have the same 
cohomology class in $H^k_{elm}(X)$. So, let $\phi_{T,f'}\in\Phi_{T}([f])$. By Proposition \ref{elem_prop} (2) we have $\psi \phi_{f',T} = f' $.
$\ \ \square $

\section {Lipschitz Retractions}\label{sec_retract}
One of the main ingredients for the proof of $L^\infty$ version of Poincar\'e lemma, Theorem \ref{Poincare}, is a Lipschitz strong deformation retraction that preserves a certain stratification. The main result of this section is 
\begin{thm}(Retraction Theorem)\label{retraction}
Let $(X,\Sigma_X)$ be a stratified set in $\R^n$ and $p\in X$ then:
\begin{enumerate}[1.]
\item There exists a stratified neighborhood $(U,\Sigma_U)$ of $p$ in $X$ such that $\Sigma_U\prec\Sigma_X\cap U$.
\item\label{L2ret} There exists $N\subset U$, $p\in N$, $\dim N< \dim U$ and Lipschitz strong deformation retraction $r:U\times [0,1]\to \ U$ to $N$ such that
  \begin{enumerate}[\ref{L2ret}.1]
    \item $\Sigma_N:=\Sigma_U\cap N$ is a stratification of $N$.
    \item $r_0(x)\in N$ and $r_1(x)=x$ where $r_t(x):=r(x,t)$ for $t\in[0,1]$.
    \item $r|_{S\times(0,1]}$ is smooth and $r(S\times(0,1])\subset S$ for any stratum $S\in     \Sigma_U$. 
    \item For any $S\in\Sigma_U$ there exists $S'\in\Sigma_N$ such that $r_0(S)\subset S'$.
    \item $d_x r_t|_S(x)\to d_x r_0|_S (x)$ as $t\to 0$, for any $x\in S\in\Sigma_U$.
   \end{enumerate}
\end{enumerate}
\end{thm}


This section is organized as follows.
In the first part we establish some useful technical tools and prove Proposition \ref{lip_infty}.
In the second part we prove Theorem \ref{retraction}.
In what follows we will use the following notation.
\begin{notation}
Suppose that $f:X\to Y$ is a map. Set $\gamma_f:X\to X\times Y$ to be the map defined 
by $\gamma_f(x):=(x,f(x))$.
\end{notation}

\begin{lem}\label{lem graph a reg hor c1}
Let  $\xi:(X,\Sigma) \to \R$ be a Lipschitz function that is smooth on every stratum $S\in\Sigma$. 
Assume that $\gamma_\xi (\Sigma)$ is a Whitney $(a)$ stratification of $\gamma_\xi (X)$. Then $\gamma_\xi$ is a semi-differentiable map with respect to $\Sigma$ and $\gamma_\xi(\Sigma)$.
\end{lem}

\begin{proof}
Let $S\in \Sigma$ be a stratum and let $x_n\in S$ be a sequence
tending to $x\in S'\in \Sigma$. 
Let $u_n \in T_{x_n}S$ be a sequence of vectors tending to a vector $u \in T_x S'$.
The vector $v_n : =(u_n;d_{x_n} \xi|_S\cdot u_n)$ (resp.
$v:=(u;d_x \xi|_{S'}\cdot u))$ is the unique vector of $T_{\gamma_\xi (x_n )} \gamma_\xi (S)$
(resp. $T_{\gamma_\xi(x)} \gamma_\xi (S')$) that projects onto $u_n$ (resp. $u$). Assume that
the sequence $v_n$ does not tend to $v$. Extracting a
subsequence, if necessary, we may assume that $v_n$ has another limit,
$v_1 \neq v$, and that the limit
$\tau :=\lim T_{\gamma_\xi (x_n )} \gamma_\xi (S)$ exists.
The vector $v$ lies in $\tau$ and is actually the unique vector that projects onto $u$.
By Whitney $(a)$ condition $T_{\gamma_\xi(x)} \gamma_\xi (S') \subset \tau$ and hence $v=v_1$.
This is a contradiction, therefore $v_n \to v$ .
\end{proof}
\begin{cor}\label{Lip_semi_diff}
Let $X\subset\R^n$ be a set and $\xi:X\to \R$ be a Lipschitz function. There exists a stratification $\Sigma$ of $X$
such that $\xi$ is semi-differentiable with respect to $\Sigma$ and the trivial stratification $\{\R\}$ of $\R$. 
\end{cor}
\begin{proof}
Let $\pi:\R^n\times\R\to\R^n$ be the standard projection to the first $n$ components
and $\pi_1:\R^n\times\R\to\R$ be the standard projection to the last component.
Let $\Sigma_\xi$ be a Whitney (A) stratification of $\Gamma_{\xi}(X)$. Set $\Sigma:=\pi(\Sigma_\xi)$.
By Lemma \ref{lem graph a reg hor c1} we know that the map $\gamma_\xi$ is semi-differentiable
with respect to $\Sigma$ and $\Sigma_\xi$. 
Note that the map $\pi_1$ is smooth on $\R^n\times\R$ and 
therefore it is semi-differentiable with respect
to $\Sigma_\xi$ and $\{\R\}$.
Since $\xi=\pi_1\circ \gamma_\xi$ it follows from Proposition \ref{comp_semi_diff} that
$\xi$ is semi-differentiable with respect to $\Sigma$ and $\{\R\}$.
\end{proof}
We turn now to prove Proposition \ref{lip_infty}.
\begin{proof}
We have already mentioned that an $L^\infty$ map is Lipschitz. For the other implication 
denote by $f_i:X\to\R$ the components of $f$, $i=1,\dots,m$. 
By Corollary \ref{Lip_semi_diff} there exist stratifications $\Sigma_i$, $i=1,\dots,m$ 
of $X$ such that each $f_i$ is semi-differentiable with respect to $\Sigma_i$ and $\{\R\}$.
Let $\Sigma_0$ be a common refinement of all the $\Sigma_i$'s.
Let $\Sigma_Y$ be a stratification of $Y$ compatible with $f(\Sigma_0)$.
Let $\Sigma_X$ to be a refinement of $\Sigma_0$ that is compatible with $f^{-1}(\Sigma_Y)$.
It follows that for each $S'\in\Sigma_X$ there exists a stratum $S\in\Sigma_Y$ such that $f(S')\subset S$. 
\end{proof}

In what follows we will use the following notations. Set $e_i$ , $i=1,...,n$ to be the standard basis
of $\R^n$ and $S^{n-1}$ to be the unit sphere of $\R^n$.
Let $\lambda \in S^{n-1}\subset \R^n$, 
denote by $N_\lambda$ the normal space to $\lambda$ in $\R^n$ and 
by $\pi_\lambda$ the projection onto
$N_\lambda$. 
Given $q \in \R^n$ , $q_\lambda$ denotes the coordinate along $\lambda$.
We say that a set $H\subset\R^{n+1}$ is a graph for $\lambda$ if there exists a function
$\xi : \R^n \to \R $ such that
$$ H=\{ q\in\R^{n+1} : q_\lambda=\xi(\pi_\lambda(q)) \} .$$

\begin{df} A {\bf Lipschitz cell decomposition of $\R^n$} is a cylindrical cell decomposition $\C$ of $\R^n$ (see Definition \ref{ccd}) which is also a stratification such that for $n>1$ each cell $C\in\C$ is either a graph of a Lipschitz function or a band delimited by two Lipschitz functions over some cell $C'$ in $\R^{n-1}$.
The vector $e_n$ is said to be {\bf regular} for $\C$ if for each cell $C\in\C$
on which $\pi_n:=\pi_{e_n}$ is one-to-one, there exists a Lipschitz function $\xi:\pi_n(C)\to \R$
such that $C$ is the graph of $\xi$ over $\pi_n(C)$.
\end{df}

The proof of Theorem \ref{retraction} relies on technique developed in [V1].
For completeness, we state some terminology and results from [V1] that will
be used in the proof.

\begin{df}\label{family regulier d hypersurface}
A {\bf regular family of hypersurfaces} of $\R^{n+1}$ is a family
$H=(H_k;\lambda_k)_{1 \leq k \leq b}$ with $b\in\N$, of subsets of
$\R^{n+1}$ together with elements $\lambda_k$ of $S^n$ such that the following
properties hold for each $k < b $:

\begin{itemize}
 \item[(i)]  The sets $H_k$ and $H_{k+1}$ are respectively the
 graphs
 for $\lambda_k$ of two global Lipschitz functions $\xi_k$ and $\xi'_k$
 such that $\xi_k \leq \xi'_k$.
\item[(ii)]  We have:
$$E(H_{k+1};\lambda_k)=E(H_{k+1};\lambda_{k+1}), $$
\end{itemize}
where 
$$E(H_k,\lambda_k)=\{ q\in\R^{n+1} : q_\lambda \leq \xi( \pi_\lambda (q)) \} \ .$$
Let $A$ be a  subset of $\R^{n+1}$ of empty interior. We say that
the family $H$ is {\bf compatible} with $A$, if $A \subseteq
\bigcup_{k=1} ^{b} H_k$. An {\bf extension} of $H$ is a regular
family compatible with the set $\bigcup_{k=1} ^b H_k$.

\end{df}

\begin{df}\label{boule  reguliere}
Let $A$ be a set of $\R^{n+1}$. An element $\lambda$ of $S^n $
is said to be {\bf regular } for $A$ if there is $\alpha >0 $
such that
$$d(\lambda;T_x A_{reg}) \geq \alpha$$
for any $x \in A_{reg}$ where $A_{reg}$ denotes the regular (smooth) part of $A$.  
\end{df}

Given two functions $f,g: A\to\R$ we say that $f$ is {\bf equivalent} to $g$,  $f\sim g$,
if there exist $c_1>0$ and $c_2>0$ such that $c_1f\leq g\leq c_2f$.
If $f \leq c_1 g$, we write $f\lesssim g$. 
We say that $f$ is {\bf comparable} with $g$ if
the difference $f-g$ has a constant sign.

\begin{thm} \label{prop_3_10} {\bf [V1]}
For each semi-algebraic set $A\subset\R^{n+1}$ of empty interior, there exists
a regular family of hypersurfaces of $\R^{n+1}$ compatible with $A$ .
\end{thm}

\begin{thm}\label{lem function eq aux distances}{\bf [V1]}
Given a function $f$ on $\R^n$, there exist a finite number of
subsets $W_1, \dots,W_s$, and a partition of $\R^n$ such that $f$
is equivalent to a product of powers of distances to the $W_j$'s
on each element of the partition.
\end{thm}

\begin{rk}\label{r1}
${\ }$
\begin{itemize}
 \item
 If $A$ is a union of graphs for a direction
 $\lambda$ of  functions $\theta_1 , \dots , \theta_k$ over $\R^n$  then we may
  find an ordered family of    functions  $\xi_1 \leq \dots
  \leq\xi_k$ such that $A$ is a union of graphs of these functions for $\lambda$.
\item
Given a family of  Lipschitz functions $f_1, \dots, f_k$
defined over $\R^n$ we can find a cell decomposition $\C'$
of $\R^n$ and some Lipschitz functions $\xi_1\leq \dots \leq  \xi_m$
on $\R^n$ such that over each cell
$C=\{ q=(x\ ;\ q_{n+1})\in\R^{n+1} :\ x\in C' ,\ \xi_i\leq q_{n+1} \leq \xi_{i+1}\}$
where $C'\in\C'$, the functions $|q_{n+1}-f_i(x)|$  are comparable with each other
and comparable with functions $f_i \circ \pi_n$.
Indeed, it suffices to consider the  graphs of
functions  $f_i$, $f_i+f_j$  and  $\frac{f_i+f_j}{2}$
and now the family $\xi_1, \dots, \xi_m$  is  given by the previous
point.
\end{itemize}
\end{rk}

\subsection{Proof of the retraction Theorem.}$\ $\\
We prove Theorem \ref{retraction} by induction on $n$ where the induction hypothesis is 
\medskip\\
$\mathbf{(H_n)}$:
Suppose that $X_1,\dots ,X_s\subset \R^n$, 
that contain $p$ in their closure and a collection of bounded functions $\xi_1,\dots,\xi_l:\R^n\to\R_+$ are given.
Then, there exist a bi-Lipschitz transformation $g:\R^n\to \R^n$,
and a stratified neighborhood $(U,\Sigma_U)$ of $p$ in $X$ compatible with 
$X_1\cap U,\dots,X_s\cap U$ 
such that 
\begin{enumerate}
\item 
$g|_S$ is a diffeomorphism for every $S\in\Sigma_U$ and $\Sigma_V:=\{g(S)\}_{S\in\Sigma}$ is a stratification of $V:=g(U)$.
\item
There exist a set $N\subset V$, $\dim N<\dim V$ and Lipschitz strong deformation retraction $r:V\times [0,1]\to V$ to $N$ 
that satisfies \ref{L2ret}.1-\ref{L2ret}.5 in the statement of the Theorem.
\item\label{xi_ineq}
$\xi_j\circ g^{-1}(r(x,t))\lesssim\xi_j\circ g^{-1}(x)$ for all $x\in V$.
\end{enumerate}
\medskip
\begin{proof}
As the statement for $n=1$ is clear we proceed to the proof 
of $\mathbf{(H_{n+1})}$ assuming $\mathbf{(H_{n})}$. 
Throughout the proof, we represent points of $\R^{n+1}$ by $q=(x,y)\in\R^{n}\times\R$.
Let $X_1,\dots,X_s\subset\R^{n+1}$ that contain $p$ in their closure and $\xi_1,\dots,\xi_l:\R^{n+1}\to\R_+$ be bounded functions.
We divide the proof into 3 steps. In the first step we reduce the problem to the case 
where all the sets are subsets of graphs of Lipschitz functions. In the second step 
we "prepare" the functions $\xi_j$ and in the final step we construct the 
bi-Lipschitz map $g$ and the Lipschitz strong deformation retraction $r$.
\\
{\bf STEP I: Reduction of the problem.}

By Theorem \ref{lem function eq aux distances} there exists a finite partition $\{V_i\}_{i\in I}$
of $\R^{n+1}$ and a finite family of subsets $\{W_{j}\}_{j\in J}$ with empty interiors (if not, we may replace them by their topological boundaries),  
such that on each element $V_i$ that contains $p$ in its closure we have:
\begin{equation}\label{prod_dist}
\xi_k(q)\sim \prod_{j\in J}{d(q,W_j)^{w_{ijk}}}\ ,\ \ q\in V_i\ ,  
\end{equation}
where $1\leq k\leq l$ and $w_{ijk}\in\Q$.
We may assume that $p\in W_{j}$ for all $j\in J$ since we can remove all $W_{j}$ that do not contain $p$ 
without affecting formula (\ref{prod_dist}).
%
%
%
%
%
By Theorem \ref{prop_3_10} there exists a regular system of
hypersurfaces $H=\hk$ compatible with the topological boundaries of $X_i$'s,
$V_i$'s and $W_j$'s.
%
%
%
%
%
%
%
Next, we reduce the problem to the case where we have a stratification $\Sigma_1$ 
of $\R^{n+1}$ compatible with the sets $X_i$, $V_i$ and $W_i$ such that all the topological boundaries of 
$X_i$, $V_i$ and $W_i$ are union of strata, which are graphs of Lipschitz functions for $e_n$ over 
strata in $\R^{n}$. The reduction is obtained by constructing a bi-Lipschitz map $h:\R^{n+1}\to\R^{n+1}$ and a stratification $\A$ of $\R^{n+1}$ such that $h|_{A}$ is a diffeomorphism for every $A\in\A$ and setting $\Sigma_1$ to be $h(\A)$.


 We define $h$
over $E(H_k;\lambda_k)$,  by induction on $k$, in such a way that
$$h(E(H_k;\lambda_k))=E(F_k ;e_n)$$ (and hence $h(H_k)=F_k$) where
$F_k$ is the graph of a Lipschitz function $\eta_k$ for $e_n$.

For $k=1$ choose an orthonormal  basis of $N_{\lambda_1}$ and set
$h(q)=(x_{\lambda_1};q_{\lambda_1})$ where $x_{\lambda_1}$
 are the coordinates of $\pi_{\lambda_1}(q)$ in this basis.
Then, let $k \geq 1$ and assume that $h$ has already been
constructed on $E(H_k;\lambda_k)$. By $(i)$ of Definition \ref{family regulier d hypersurface} 
the sets $H_k$ and
$H_{k+1}$ are the graphs for $\lambda_k$ of two Lipschitz
functions $\zeta_k$ and $\zeta'_k$. For $q \in E(H_{k+1};\lambda_{k})
\setminus E(H_k;\lambda_k)$ define $h(q)$ to be  the element :
$$h(\pi_{\lambda_k}(q);\zeta_k \circ \pi_{\lambda_k}(q))
+(q_{\lambda_k}-\zeta_k \circ \pi_{\lambda_k}(q))e_n.$$

Thanks to the property $(ii)$ of  Definition \ref{family regulier
d hypersurface} we have $ E(H_{k+1};\lambda_{k+1})=
E(H_{k+1};\lambda_{k})$, so that   $h$ is actually defined over $
E(H_{k+1};\lambda_{k+1})$. Since $\zeta_k$ is Lipschitz this is a
bi-Lipschitz homeomorphism. Note also that the image is
$E(F_{k+1};e_n)$ where $F_{k+1}$ is the graph of the Lipschitz
function
$$\eta_{k+1}(q)=\eta_k \circ \pi_{e_n}(q)+(\zeta'_k-\zeta_k)
\circ \pi_{\lambda_k}\circ h^{-1}(q;\eta_k \circ \pi_{e_n}(q)).$$
This gives $h$ over $E(H_b ; \lambda_b)$. To extend $h$ to the
whole of $\R^n$ do it as in the case $k=1$ (use $\lambda_b$ instead of $\lambda_1$). Now it is easy to
check that this defines a bi-Lipschitz homeomorphism.

Next, we construct a stratification $\A$ of $\R^{n+1}$ such that 
$h|_A$ is a diffeomorphism for every $A\in \A$ and moreover,
$h(A)$ is either included in a graph of one of the $\eta_i$'s or is a band
delimited by the graphs of two consecutive $\eta_i$'s.

By induction on $k$ we define a family of stratifications $\F_k:=\{\A_{1,k},\dots,\A_{k,k} \}$ 
such that for each $i$,  $\A_{i,k}$  is a stratification of $H_i$ that refines $\A_{i,k-1}$. 

For $k=1$ define $\A_{1,1}$ to be a stratification of $H_1$.
Suppose that $\F_k$ was constructed we construct $\F_{k+1}$ as follows.

Define $\A_{k+1-j,k+1}$ by induction on $j$.
For $j=0$ set $\A_{k+1,k+1}$ to be a stratification of $H_{k+1}$.
Suppose that $\A_{k+1-j,k+1}$ was constructed we construct $\A_{k-j,k+1}$.
Since the hypersurface $H_{k+1-j}$ is a graph of a Lipschitz function for $\lambda_{k-j}$,
we set $\A_{k-j,k+1}$ to be a refinement of $\A_{k-j,k}$ that is compatible with all $\pi_{\lambda_{k-j}}^{-1}(A)\cap H_{k-j}$ for $A\in\A_{k-j+1,k+1}$.

Now, the family $\F_b$ consists of stratifications of the hypersurfaces\\ $\hk$
that induces a stratification $\A$ of $\R^{n+1}$ in the following way. 
The strata of $\A$ are:
\begin{itemize} 
\item The strata of each $\A_{j,b}$, $j=1,\dots,b$
\item The bands delimited by the graphs of $\zeta_k$ and $\zeta'_k$ for $\lambda_k$ intersected with 
$\pi_{\lambda_k}^{-1}$(A) where $A\in \A_{k,b}$ and $k=1,\dots,b$.
\item $\{q:q_{\lambda_{b}}>\zeta_b(q)\}$ and $\{q:q_{\lambda_{1}}<\zeta_1(q)\}$
\end{itemize}
By the construction of $h$, it is evident that $h|_A$ is smooth for all $A\in\A$ and 
$\Sigma_1:=\{h(A)\}_{A\in\A}$ forms a stratification of $\R^{n+1}$.

%
%
%
%
%
%

Note that the projection of $\Sigma_1$, $\Sigma_1':=\pi_n(\Sigma_1)$ forms a stratification of $\R^n$
that is
compatible with $\{\pi_n (h(W_{j}))\}_{j\in J},\ \{\pi_n (h(V_i))\}_{i\in I},\ \{\pi_n (h(X_i))\}_{1\leq i\leq s}$ and  the restrictions of all the $\eta_j$'s to the strata of $\Sigma_1'$ are smooth.

We may identify $h$ with the identity map. Indeed, suppose that 
we proved $\mathbf{(H_{n+1})}$ for the sets in $\Sigma_1$ and the functions  
$\xi_j\circ h^{-1}$ obtaining a bi-Lipschitz map 
$g_h:\R^{n+1}\to\R^{n+1}$, a stratified neighborhood $(U_h,\Sigma_{U_h})$ and a Lipschitz strong deformation retraction $r_h : V_h\times I \to V_h$  to a set $N_h\subset V_h$, where $V_h:=h(U_h)$, that satisfy the conclusion of $(H_{n+1})$. 
Set
$$ g:=g_h\circ h,\  U:=h^{-1}(U_h)\ \text{ and }\ \Sigma_U:=h^{-1}(\Sigma_{U_h}). $$
Note that $\Sigma_{U_h}$ is compatible with the sets in $\Sigma_1\cap U_h$ and therefore 
$h^{-1}$ is still a diffeomorphism on the strata of $\Sigma_{U_h}$.
It follows that $g|_S$ is a diffeomorphism for every $S\in\Sigma$ 
and $\{g(S)\}_{S\in\Sigma}$ is a stratification of $V=V_h=g(U)$.
So we set $r:=r_h$ and $N:=N_h$, which trivially satisfy conditions \ref{L2ret}.1-\ref{L2ret}.5 in the statement of the Theorem. The last condition of $\mathbf{(H_{n+1})}$ is obviously preserved.
\\ 
{\bf STEP II: Preparation of functions $\xi_j$.}

The aim of this step is to find a refining stratification $\C$ of $\Sigma_1$ such that 
over each cell $C\in\C$ one has
\begin{equation}\label{xi_m_prepared}
\xi_k(q)\sim |y-\eta_{\nu_k}(x)|^{w_{k}}a_k(x)\ , 
\end{equation}
where $(a_k) 's$ are functions to be specified later and $\nu_k,{w_{k}}$, ${w_{\nu,\nu'}}$, and ${w'_{jk}}$ are constants.

Note that 
\begin{equation}\label{split_sum}
d(q, W_{j}\cap\Gamma_{\eta_\nu})\sim |y-\eta_\nu(x)|+
 d(x, \pi_n(W_{j}\cap\Gamma_{\eta_\nu}))\ ,\ \  q\in\R^{n+1}\ , 
\end{equation}
where $j\in J$ and $1\leq\nu\leq b$. 
%
%
By Remark \ref{r1}, there exists a collection of Lipschitz functions  
$\{\theta_1,\dots,\theta_{b'}\} \supset \{\eta_1,\dots,\eta_b\}$ on $\R^n$ such that 
there exists a cell decomposition $\C_0$ of $\R^{n+1}$ with the following properties.
\begin{enumerate}
\item The cells of $\C_0$ are obtained from $\Sigma_1$ by adding the graphs and bands of the $\theta_i$'s over the cells of $\R^n$.
\item The functions $d(x, \pi_n(W_j\cap\Gamma_{\eta_\nu}))$, 
$\eta_\nu$, $|\eta_\nu-\eta_{\nu'}|$ , and $|y-\eta_\nu|$, where $1\leq\nu,\nu'\leq b$ and $j\in J$, are pairwise comparable with each other.
\end{enumerate}
Let $\C$ be a stratification of $\R^{n+1}$ that is obtained from $\C_0$ by refining the cells in $\R^n$ in such a way  that the stratification of $\R^{n+1}$, resulting by taking graphs and bands of the restrictions of $\theta_j$'s to those cells, forms a Whitney (a) stratification.


Let $C$ be an open cell of $\R^{n+1}$ that is delimited by the graphs of $\theta_{j_0}$ and
$\theta_{j_0+1}$ over a cell $C'$ of $\R^n$. Due to the fact that the cell decomposition $\C$ is compatible with the graphs of the $\eta_i$'s, 
we have either $\eta_i|_{C'}\geq\theta_{j_0+1}$ or $\eta_i|_{C'}\leq\theta_{j_0}$ for any $i\in\{1,\dots,b\}$.
Note that for any $j\in J$, 
$$ d(q,W_j)=\min_\nu d(q,W_j\cap\Gamma_{\eta_\nu})\ . $$
So, by (\ref{prod_dist}) and (\ref{split_sum}) there exists $i$ such that on $V_i$ we have
\begin{eqnarray*}
\xi_k(q) &\sim& \prod_{j\in J} (\min_\nu d(q,W_j\cap\Gamma_{\eta_\nu}))^{w_{ijk}}\\ 
&\sim&
\prod_{j\in J} \left(|y-\eta_\nu(x)|+ d(x, \pi_n(W_j\cap\Gamma_{\eta_\nu}))\right)^{w_{ijk}}\ .  
\end{eqnarray*}
Each expression of the form
$$ \left(|y-\eta_\nu(x)|+
d(x, \pi_n(W_{j}\cap\Gamma_{\eta_\nu}))\right)^{w_{ijk}} $$
is equivalent to
\begin{equation}\label{a_le_0}
 \min\left( |y-\eta_\nu(x)|^{w_{ijk}},
d(x, \pi_n(W_{j}\cap\Gamma_{\eta_\nu}))^{w_{ijk}} \right)
\quad\text{   if ${w_{ijk}}<0$ } 
\end{equation}
and
equivalent to  
\begin{equation}\label{a_ge_0}
 \max\left( |y-\eta_\nu(x)|^{w_{ijk}},
d(x, \pi_n(W_{j}\cap\Gamma_{\eta_\nu}))^{w_{ijk}} \right)
\quad\text{   if ${w_{ijk}}>0$\ . } 
\end{equation}

Since over the cell $C$, the functions $|y-\eta_\nu(x)|$, $|\eta_\nu-\eta_{\nu'}|$ and 
\mbox{$d(x, \pi_n(W_{j}\cap\Gamma_{\eta_\nu})) $}
are pairwise comparable with each other for all $1\leq\nu,\nu'\leq b$, $i\in I$, $j\in J$ and $1\leq k\leq s$, the expressions in (\ref{a_le_0}) and (\ref{a_ge_0}) 
are equal to either $|y-\eta_\nu(x)|^{w_{ijk}}$ or 
$d(x, \pi_n(W_{j}\cap\Gamma_{\eta_\nu}))^{w_{ijk}}$ on $C$.
Also, one of the following 3 holds
\begin{itemize}
\item $|y-\eta_\nu(x)|\sim |y-\eta_{\nu'}(x)|$ 
\item $|y-\eta_\nu(x)|\sim |\eta_\nu(x)-\eta_{\nu'}(x)|$
\item $|y-\eta_{\nu'}(x)|\sim |\eta_\nu(x)-\eta_{\nu'}(x)|$ . 
\end{itemize}
It follows from here that, there exist constants $\nu_k,{w_{k}}$, ${w_{\nu,\nu'}}$, and ${w'_{jk}}$
such that over the cell $C$ formula (\ref{xi_m_prepared}) holds with
$$a_k(x)=\prod_{\nu,{\nu'}}|\eta_\nu-\eta_{\nu'}|^{w_{\nu,\nu'}}
\prod_{j\in J} d(x, \pi_n(W_{j}\cap\Gamma_{\eta_\nu}))^{w'_{jk}}\ .$$
%
%
%
%
%
%
%
%
%
%
\\
{\bf STEP III: Construction of the map $g$ and the deformation retraction $r$.}

Apply the induction hypothesis to the cells of $\C':=\pi_n(\C)$ in $\pi_n(\R^{n+1})=\R^{n}$
that contain $p':=\pi_n(p)$ in their closure and to the following collection of functions 
\begin{itemize}
\item $|\theta_j(x)|$ for $1\leq j\leq b'$,
\item $|\eta_j(x)-\eta_{j+1}(x)|$ for ${1\leq j\leq b-1}$,
\item $|\theta_j(x)-\theta_{j+1}(x)|$ for ${1\leq j\leq b'-1}$,  
\item $\min\left(a_k(x)|\theta_j(x)-\theta_{j+1}(x)|^{w_{k}},1\right)$ 
for ${1\leq j\leq b'-1}$, ${1\leq k\leq l}$ , 
\item $\min\left(a_k(x)|\theta_j(x)-\eta_{\nu_k}(x)|^{w_{k}},1\right)$ 
for ${1\leq j\leq b'}$,  ${1\leq k\leq l}$,
\item $\min\left(a_k(x)|\theta_{j+1}(x)-\eta_{\nu_k}(x)|^{w_{k}},1\right)$ 
for ${1\leq j\leq b'-1}$, ${1\leq k\leq l}$. 

\end{itemize}
This provides a bi-Lipschitz map $g':\R^n\to\R^n$ and stratified neighborhood $(U',\Sigma')$ compatible with the cells of $\C'$ such that $g'|_S$ is a diffeomorphism for all $S\in\Sigma'$ and $\{g(S)\}_{S\in\Sigma'}$ is a stratification of $V':=g'(U')$.
Also, this provides a set $N'\subset V'$, $\dim N'< \dim V'$ and
a Lipschitz deformation retraction $r':V'\times I\to V'$ to $N'$ that preserves 
the strata of $\Sigma'_V:=g'(\Sigma')$.\\
Define $g:\R^{n+1}\to\R^{n+1}$ by
$$ g(q)= (g'(\pi_n(q)),q\cdot e_{n+1}), $$
where $'\cdot'$ denotes the standard scalar product.\\
Define $$ U:=\bigcup_{C} g^{-1}(\pi_n^{-1}(U')\cap C), $$
where $C$ runs over all the cells in $\C$ that contain $p$ in their closures.
\\
Define  $$ N:=N'\times\R\cap U\ .$$
To construct $\Sigma=\Sigma_U$ we note that $\Sigma'$ is a refinement of $\C'$. 
Therefore, $\Sigma'$ induces a cylindrical cell decomposition $\Sigma$ that refines $\C$ in the following way. The strata of $\Sigma$ are the strata of $\Sigma'$ and graphs and bands of functions that define the cells of $\C$ restricted to the cells in $\Sigma'$. 
It follows from the construction of $g$ and $\Sigma$ that $g|_S$, $S\in\Sigma$ is a diffeomorphism.
For the rest of the proof we will identify $\Sigma$ with its image by $g$.
The lift of $r'$ is defined as follows. Let $C$ be a cell in $\Sigma$. The cell $C$ 
can be either a graph of a function $\theta_j$ over a cell of $\Sigma'$ or a the 
set between two consecutive graphs of $\theta_j$ and $\theta_{j+1}$ over a cell of $\Sigma'$. 
In the former case the lift of $r'$ to $C$ is defined by 
$$ r(q,t)=(r'(x,t),\theta_j(r'(x,t))), $$
where $q=(x,y)\in\R^n\times\R$.
In the other case, we represent $y$ as \\
\mbox{$\tau(q)\theta_{j+1}(x)+(1-\tau(q))\theta_{j}(x)$} where $\tau:C\to[0,1]$ is defined 
by $$\tau(q):=\frac{y-\theta_j(x)}{\theta_{j+1}(x)-\theta_j(x)}\ . $$
Set 
\begin{equation}\label{set_r}
r(q,t):=(r'(x,t),\tau(q)\theta_{j+1}(r'(x,t))+(1-\tau(q))\theta_{j}(r'(x,t)))\ . 
\end{equation}
We have to show that $r$ is Lipschitz, $(d_xr_t|_C)(q) \to (d_x r_0|_C)(q)$ as $t\to 0$ and that
condition (\ref{xi_ineq}) of $\mathbf{(H_{n+1})}$  holds.

\textit{Proof that condition (\ref{xi_ineq}) of $\mathbf{(H_{n+1})}$  holds.}
%
%
In the case that $C$ is a graph over a cell of $\Sigma'$ condition (\ref{xi_ineq}) follows easily from the induction hypothesis. In the other case, the cell $C$ is delimited by $\theta_j$  and $\theta_{j+1}$ and we assume that each $\xi_k$, $1\leq k\leq s$ , is of the form (\ref{xi_m_prepared}) and either $\eta_{\nu_k}\leq\theta_j$ or $\eta_{\nu_k}\geq\theta_{j+1}$. Let us assume,
without loss of generality that the former case holds. 
Thus, 
$$|y-\eta_{\nu_k}(x)|=|y-\theta_j(x)|+|\eta_{\nu_k}(x)-\theta_j(x)|\ .$$
Set 
$$ \theta(x):=|\theta_{j+1}(x)-\theta_j(x)| \text{ and } \eta(x):=|\eta_{\nu_k}(x)-\theta_j(x)|\ .$$
Therefore, 
\begin{equation}\label{xi_mq}
\xi_k(q)\sim a_k(x)\left\{ \begin{array}{ll}
                \min(|y-\theta_j(x)|^{w_{k}},
                \eta(x)^{w_{k}} )  &  w_{k} <0\\
                \max (|y-\theta_j(x)|^{w_{k}},
                \eta(x)^{w_{k}} )  &  w_{k} >0\ .
  
               \end{array} \right.
\end{equation}
Let $z:=z(t)=(z_1(t),z_2(t))$ be the components of the retraction $r=r(q,t)$ where $(z_1(t),z_2(t))\in \R^n\times\R$, where $q=(x,y)$. To simplify the notation we will write $(z_1,z_2)$ instead of $(z_1(t),z_2(t))$.
From (\ref{set_r}) and (\ref{xi_mq}) it follows that,
\begin{equation}\label{xi_m1}
\xi_k(z)= a_k(z_1)\left\{ \begin{array}{ll}
                 \min\{|\tau(z)\theta(z_1)|^{w_{k}},
                 \eta(z_1)^{w_{k}}\}  &  w_{k} <0\\                
                 \max\{|\tau(z)\theta(z_1)|^{w_k},
                 \eta(z_1)^{w_{k}} \} & w_{k} >0\ .\\
           \end{array} \right.
\end{equation}
Note that $\tau(q)=\tau(z)$.
\newline
If ${w_k}<0$ then, since $\xi_k$ is bounded, 
\begin{multline}\label{xi_m_le0}
\xi_k(z) \sim \min\{ \min(a_k(z_1)|\tau(z)\theta(z_1))|^{w_k},1),
\min(a_k(z_1)\eta(z_1)^{w_k},1) \}\ .
\end{multline}

We remark that if condition (\ref{xi_ineq}) of the induction hypothesis holds for $f_1$ and $f_2$ then it also
holds for $\min\{f_1,f_2\}$. Also note that if $f$ is a non-negative and bounded function then $f\sim \min(f,1)$.
Therefore, it is enough to prove that 
$$\min(a_k(z_1)|\tau(z)\theta(z_1))|^{w_k},1)\lesssim 
\min(a_k(x)|\tau(q)\theta(x)|^{w_k},1)$$
and 
$$\min(a_k(z_1)\eta(z_1)^{w_k},1) \lesssim
\min(a_k(x)\eta(x)^{w_k},1)\ . $$
The latter immediately follows from the induction hypothesis. 
For the former inequality we note that by induction hypothesis
$$
\min(a_k(z_1)\theta(z_1)^{w_k},1)\lesssim
\min(a_k(x)\theta(x)^{w_k},1)          
$$
and therefore,
\begin{eqnarray}\label{xi_m2}
\min(a_k(z_1)|\tau(z)\theta(z_1))|^{w_k},1) &=&
\min\{\tau(z)^{w_k} a_k(z_1)\theta(z_1)^{w_k},\tau(z)^{w_k},1\}\nonumber\\&=&
\min\{\tau(z)^{w_k}\min( a_k(z_1)\theta(z_1)^{w_k},1),1\}\nonumber\\&\lesssim&    
\min\{\tau(q)^{w_k} a_k(x)\theta(x)^{w_k},1\}\nonumber\ .         
\end{eqnarray}

Suppose now that ${w_k}>0$. It follows from the fact that $\xi_k$ is bounded, formula (\ref{xi_m1}) and the 
induction hypothesis that
\begin{equation}\label{xi_mge0}
a_k(z_1)|\tau(z)\theta(z_1)|^{w_k} \lesssim
a_k(x)|\tau(q)\theta(x)|^{w_k}\ .
\end{equation}
Therefore,
\begin{eqnarray*}
\xi_k(z) &\sim&
\max\{ (a_k(z_1)|\tau(z)\theta(z_1)|^{w_k},
                 a_k(z_1)\eta(z_1)^{w_k} \} \\
                 &\lesssim&
\max(a_k(x)|\tau\theta(x)|^{w_k},
                 a_k(x)\eta(x)^{w_k} )\\ 
                 &\sim&     \xi_k(q)\ .
\end{eqnarray*}
%
%
%

\textit{Proof that $r$ is Lipschitz and $(d_x r_t|_C)(q)\to (d_x r_0|_C)(q)$ as $t\to 0$.} 
To prove that $r$ is Lipschitz we show that the differential of $r$ is bounded where defined. In the case that the cell $C$ is a graph of $\theta_j$
over a cell $C'$ of $\Sigma'$, it follows immediately from the induction hypothesis that $r$ is Lipschitz. 

To see that $(d_x r_t|_C)(q)\to (d_x r_0|_C)(q)$ as $t\to 0$, fix $q=(x,\theta_j(x))\in C$. Recall that 
the graph of $\theta_j$ is Whitney (a) stratified over $\C'$ and therefore, by Lemma \ref{lem graph a reg hor c1},
$\theta_j$ is semi-differentiable with respect to $\C'$ and therefore semi-differentiable with respect to $\Sigma'$. Using the chain rule we compute (to simplify the notations we use the symbol '$d$' inplace of '$d_x$')
$$ (d r_t|_C)(q)=(dr'_t(x), d(\theta_j(r'_t(x))))=(dr'_t(x), d\theta_j(r'_t(x))dr'_t(x) ).  $$
Let $u\in T_{x} C'$ 
$$ \left((d r_t|_C)(q)\right)u = 
((dr'_t(x))u, d\theta_j(r'_t(x))(dr'_t(x))u ).
$$
By induction hypothesis we  have  $(dr'_t(x))u\to dr'_0(x)u$ as $t\to 0$
and since $\theta_j$ is semi-differentiable we also have 
$$ d\theta_j(r'_t(x))(dr'_t(x))u \to d\theta_j(r'_0(x))(dr'_0(x))u\ \text{ as } t\to 0. $$ 
It follows that $(d r_t|_C)(q)\to (d r_0|_C)(q)$ as $t\to 0$.

In the other case, when $C$ is a band, the derivatives of the first $n$ components of $r$
are bounded by induction hypothesis so we only have to consider the derivative of the last component.
For the sake of this computation we may assume that
$\theta_j\equiv 0$ since if not we may consider a bi-Lipschitz change of coordinates $x\mapsto x $, 
$y\mapsto y-\theta_j(x)$. Set $f:=\theta_{j+1}$.
With those assumptions, the last component of the map $r(x,y,t)$ is given by $$r_{n+1}(x,y,t)=\frac{y}{f(x)}f(r'(x,t))\ ,\quad
0 \leq y \leq f(x).$$
We estimate the derivatives of $r_{n+1}$:
\begin{eqnarray}\label{der_est}
\pa_y \left( y\frac{f(r'(x,t))}{f(x)}\right) &=& \frac{f(r'(x,t))}{f(x)} \nonumber \\
 &\lesssim& 1 \quad \text{ (by induction hypothesis) ,}\nonumber\\
\nonumber\\
\left| \pa_{x_j} \left( y\frac{f(r'(x,t))}{f(x)}\right)\right| &=& \left| y\left[\frac{\pa_{x_j}\left( f(r'(x,t))\right)f(x)-f(r'(x,t))f_{x_j}(x)  } {f^2(x)}\right]\right|\nonumber\\ 
&\lesssim&
\left|\pa_{x_j}\left( f(r'(x,t))\right) \right| + 
\left| f_{x_j}(x)\right|\nonumber \\ &\lesssim& 1\ .
\end{eqnarray}
Similarly $\pa_t r_{n+1}$ is bounded.
To show that $(d_x r_t|_C)(q)\to (d_x r_0|_C)(q)$ as $t\to 0$ it is enough to show that 
$d_{(x,y)} r_{n+1}(x,y,t) \to d_{(x,y)} r_{n+1}(x,y,0) $ as $t\to 0$.
$$ d_{(x,y)} r_{n+1}(x,y,t)=\frac{f(r'(x,t))}{f(x)}dy + y\frac{f(x)d(f(r'(x,t)))-f(r'(x,t))df(x)}{f^2(x)} .  $$ 
Note that for fixed $x$,  $f(r'_t(x))\to f(r'_0(x))$ and as in the case where $C$ is a graph, we also have $d_x(f(r'(x,t)))\to d_x(f(r'(x,0))) $ as $t\to 0$.

\end{proof}


\section{Proof of the Poincar\'e Lemma via Regularization of stratified forms }\label{smoothing}
The purpose of this section is to prove Theorem \ref{Poincare}.
We begin by proving a somewhat weaker version of the Poincar\'e Lemma.

\begin {lem} \label{weakPoincare}
Let $\omega$ be a closed smooth $L^\infty$ $k$-form on $X\subset\R^n$ and $p\in X$.There exists a neighborhood $U_p$ of $p$ in $X$, and an $L^\infty$ $(k-1)$-form $\gamma$ defined on $U_p$ such that $\omega=\overline d\gamma$ in $U_p$.
\end{lem}
\begin{proof}
The proof is by induction on $\dim X$. In the case that $\dim X = 1$, the Theorem is just the Fundamental Theorem of Calculus. Suppose that $\dim X>1$. 
Let $(\omega,\Sigma)$ be a stratified form.
By Theorem \ref{retraction}
there exist a stratified neighborhood $(U_p,\Sigma')$ of $p$ such that $(\omega|_{U_p}, \Sigma')$ is a stratified form, 
$N\subset U_p$ and $r:U_p\times I\to U_p$ a Lipschitz strong deformation retraction of $U_p$ to $N$, that preserves the strata of $\Sigma'$.
The pull back of $\omega$, $r^*\omega$ is a smooth 
$L^\infty$ $k$-form on $U_p\times I$ that can be represented
as $\alpha+dt\wedge\beta$. Note that since $r$ preserves the strata in $\Sigma'$ the forms $(\alpha,\Sigma'\times \Sigma_I)$ and  $(\beta,\Sigma'\times \Sigma_I)$ are stratified forms on $U_p\times I$, where $\Sigma'\times \Sigma_I$ denotes the stratification of $U_p\times I$ obtained by taking cross products of 
strata in $\Sigma'$ with any of $\{0\}, \{1\}$ and $(0,1)$.
Set $$ \gamma_0:=\int_0^1\beta(x,t) dt\ . $$
We claim that  $$\overline d\gamma_0=r_1^*\omega-r_0^*\omega $$ 
We will show that for any stratum $S$ and any smooth $(n-k)$-form with compact support in $S$ we have $$ \int_S \overline d \gamma_0\wedge\phi=\int_S(r_1^*\omega-r_0^*\omega)\wedge\phi. $$
\begin{eqnarray}
\int_S \overline d\gamma_0\wedge\phi
=
(-1)^k\int_S \gamma_0\wedge d \phi 
&=&
(-1)^k\int_{S\times I} r^*\omega \wedge d\phi
\nonumber \\ &=& 
\int_{S\times I} d(r^*\omega \wedge \phi)
\nonumber\\ &=&
\lim_{\eps\to 0}\int_{S\times (\eps,1]} d(r^*\omega \wedge \phi)
\nonumber \\&=& 
\lim_{\eps\to 0}\int_{\pa(S\times (\eps,1])} r^*\omega \wedge \phi
\nonumber \\&=& 
\int_S r_1^*\omega \wedge \phi - \lim_{\eps\to 0}\int_{S} r_{\eps}^*\omega \wedge \phi.
\end{eqnarray}
Since $d_x r_t|_S\to d_x r_0|_S $ it follows that 
$$ \lim_{\eps\to 0}\int_{S} r_{\eps}^*\omega \wedge \phi= \int_{S} r_{0}^*\omega \wedge \phi. $$
Note that $r_1^*\omega=\omega$ and $r_0^*\omega|_N=\omega|_N$.
By induction hypothesis for the set $N$ and the form $\omega|_N$ we obtain a form $\gamma'$ on $N$ such that  $\omega|_N=\overline d\gamma'$. 
But then, since $r_0$ is a bounded Lipschitz map we have
$$r_0^*\omega=\overline d r_0^*\gamma' .$$
Thus, we set $\gamma:=\gamma_0+r_0^*\gamma'$ and obtain 
$$ \overline d\gamma=\overline d (\gamma_0+r_0^*\gamma')=\omega \ .$$
\end{proof}

In the rest of this section we will develop tools to "smoothen" the form $\gamma$ that 
was constructed in Lemma \ref{weakPoincare}. 
We adopt an approach introduced by B. Youssin in [Y]. 
Our method is an extension of the method in [Y] to the setting of
stratified sets and $L^\infty$ forms (see subsection 
\ref{approx_strat}).  

Let $X$ be a compact Riemannian manifold with boundary $\pa X$.\\
Let $(\Omega_{DR}^k(int\ X),d)$ be the cochain complex of smooth $k$-forms on $int\ X:=X-\pa X$
and $(\overline{\Omega}_{DR}^k(int\ X),\overline d)$ be the cochain complex of 
continuous and bounded weakly differentiable forms with continuous and bounded weak differentials. Denote the cohomology of $\Omega_{DR}^k(int\ X)$ by  $H^k(\Omega_{DR}^\bullet(int\ X))$  and
by $H^k(\overline\Omega_{DR}^\bullet(int\ X))$ the cohomology of $\overline\Omega_{DR}^k(int\ X)$.

\begin{lem}\label{appr_1} 
Let $\omega\in \overline\Omega_{DR}^k(int\ X)$. There exists a family of smooth $k$-forms $\omega_\eps$, $\eps>0$ such that 
\begin{enumerate}
\item $\omega_\eps(x)\to\omega(x)$ for every $x\in X$ as $\eps\to0$. 
\item $d\omega_\eps(x)\to \overline d\omega(x)$ for every $x\in X$ as $\eps\to0$. 
\end{enumerate}
\end{lem}  
The proof of the lemma relies on classical smoothening techniques which we introduce prior to proving the lemma. 
Let $\alpha\in\Omega^k_{DR}(\R^n)$ and $\beta \in\Omega^j_{DR}(\R^n)$ be forms with compact supports.
Suppose that
 $$\alpha = \sum_K a_K(x) dx_K, $$
where $K$ runs over all multi-indices of size $k$ and
$dx_K=dx_{K_1}\wedge\dots\wedge dx_{K_k}$.
Similarly, suppose that
$$ \beta = \sum_J b_J(x) dx_J, $$
where $J$ runs over all multi-indices of size $j$.
The {\bf convolution of $\alpha$ and $\beta$} with respect to the standard coordinates 
of $\R^n$ is defined by
$$ \alpha*\beta:= \sum_{I,J} (a_I*b_J)(x) dx_I\wedge dx_J, $$
where $$ (a_I*a_J)(x):=\int_{\R^n}a_I(x)b_J(x-y)dy. $$
Define also 
$$ \tilde\alpha(x):=\alpha(-x)=\sum_I a_I(-x)dx_I .$$
The important properties of convolution of functions extend to the case of differential forms.
We summarize them in the following proposition.
\begin{prop}\label{conv_prop}
Suppose that $\alpha\in\Omega^k_{DR}(\R^n)$, $\beta \in\Omega^j_{DR}(\R^n)$ and 
$\gamma \in\Omega^l_{DR}(\R^n)$ have compact supports and $k+j+l=n$. Then, \\
(1) $\alpha*\beta=(-1)^{jk}\beta*\alpha$ .\\
(2) $d(\alpha*\beta)=(d\alpha)*\beta=(-1)^k\alpha*d\beta$.\\
(3) $\int_{\R^n} (\alpha*\beta)\wedge\gamma = \int_{\R^n} \alpha\wedge(\tilde \beta*\gamma)$.\\
(4) $d\tilde\alpha=-\widetilde{ d\alpha}$.\\
(5) Suppose that $\phi_\eps$, $\eps> 0$ is a family of smooth functions such that
$\text{supp }\phi_\eps\subset B(0,\eps)$, $\int_{\R^n}\phi_\eps(x)dx_1\dots dx_n=1$  and $\phi_\eps=\tilde\phi_\eps$. Then,
$\|\alpha  - \alpha*\phi_\eps \|\to 0$ as $\eps\to 0$.
\end{prop}
\begin{proof}
Suppose that  $\alpha=\sum_{K} a_K dx_K$, $\beta=\sum_{J} b_J dx_J$ and $\gamma=\sum_{L} a_L dx_L$.\\
{\it Proof of (1): }
$$ \alpha*\beta=\sum_{K,J} (a_K*b_J) dx_K\wedge dx_J=(-1)^{jk}\sum_{I,J} (b_J*a_K) dx_J\wedge dx_K=(-1)^{jk}\beta*\alpha.$$ 
{\it Proof of (2): }
Observe that
$$d(\alpha*\beta)=d\left(\sum_{K,J} (a_K*b_J) dx_K\wedge dx_J\right)=
\sum_{K,J,i} (\frac{\pa a_K}{\pa x_i}*b_J) dx_i\wedge dx_K\wedge dx_J=(d\alpha)*\beta. $$
This computation together with (1) shows that 
$$(-1)^{k(j+1)} \alpha*d\beta= (d\beta)*\alpha=d(\beta*\alpha)=(-1)^{jk}d(\alpha*\beta)=(-1)^{jk}(d\alpha)*\beta. $$
(2) follows.\\
{\it Proof of (3):}
In the following computation we set $dy_{[n]}:=dy_1\wedge\dots\wedge dy_n$.  
\begin{eqnarray*} 
& &\int_{\R^n} \sum_{K,J,L} (a_K*b_J)(x) c_L(x) dx_K\wedge dx_J \wedge dx_L 
\\&=& 
\int_{\R^n} \sum_{K,J,L} \left(\int_{\R^n} a_K(y)b_J(x-y)dy_{[n]}\right) c_L(x) dx_K\wedge dx_J \wedge dx_L 
\\&=& 
\sum_{K,J,L} \int_{\R^n}\int_{\R^n} a_K(y)b_J(x-y)c_L(x) dy_{[n]}dx_K\wedge dx_J \wedge dx_L 
\end{eqnarray*}
interchanging the order of integration we get
\begin{eqnarray*}
& &\sum_{K,J,L} \int_{\R^n}\int_{\R^n} a_K(y)b_J(x-y)c_L(x) dx_{[n]}dy_K\wedge dy_J \wedge dy_L 
\\&=& 
\sum_{K,J,L} \int_{\R^n} a_K(y)(\tilde b_J*c_L)(y) dy_K\wedge dy_J \wedge dy_L 
\\&=& \int_{\R^n} \alpha\wedge(\tilde\beta*\gamma).
\end{eqnarray*}
{\it Proof of (4):} Trivial.\\
{\it Proof of (5):}  Note that
\begin{equation}
 \alpha-\alpha*\phi_\eps = \sum_I \left(a_I(x)-a_I*\phi_\eps(x) \right) dx_I. 
\end{equation}
Since $a_I$'s are continuous functions,  classical arguments imply that $\| a_I - a_I*\phi_\eps\|\to 0 $ as $\eps\to 0$. Therefore, 
$$ \|\alpha-\alpha*\phi_\eps\| = \|\sum_I \left(a_I(x)-a_I*\phi_\eps(x) \right) dx_I\| \leq
\sum_I \|a_I-a_I*\phi_\eps\|\to 0 \text { as } \eps\to 0. $$

\end{proof}
Next we proceed to the proof of Lemma \ref{appr_1}.
\begin{proof}
Let $\{U_i\}$, $i\in L:=\{1,\dots,L\}$ be a cover of $X$ by coordinate charts. Let $\phi_i$ be a partition of unity subordinate to the cover $\{U_i\}_{i\in L}$. 
Since $\omega = \sum _{i\in L} \phi_i \omega$, it is enough to prove the lemma for a form supported in one coordinate chart. Therefore, we may assume that $X=\R^n$, $\omega = \sum_I a_I(x) dx_I $, where
$x$ denotes the standard coordinates in $\R^n$ and $a_I$'s are continuous functions with compact supports.
Let $\phi_\eps$ be a family of smooth functions as in Proposition \ref{conv_prop} (5) and 
set 
$$ \omega_\eps :=\omega*\phi_\eps .$$
It follows from Proposition \ref{conv_prop} (5) that $\omega_\eps(x)\to\omega(x)$ for all $x$. 
To show (2) it is enough to show  $ d\omega_\eps= \overline d \omega * \phi_\eps $, since 
$(\overline d\omega*\phi_\eps) (x)\to \overline d \omega(x) $.
We show that 
$$ \int_{\R^n}\left( \overline d \omega * \phi_\eps \right)\wedge\varphi = 
\int_{\R^n} d\omega_\eps\wedge\varphi ,$$
for any smooth $(n-k-1)$-form with compact support. Indeed,
\begin{eqnarray*}
\int_{\R^n}\left( \overline d \omega * \phi_\eps \right)\wedge\varphi &=&
\int_{\R^n}\overline d \omega\wedge\left( \varphi*\phi_\eps\right)
\\&=& 
(-1)^{n}\int_{\R^n} \omega\wedge \left(\varphi*d\phi_\eps \right)
\\&=&
(-1)^{n+nk}\int_{\R^n}(\varphi*d\phi_\eps)\wedge\omega\\ &=& (-1)^{n+nk}\int_{\R^n}\varphi\wedge(\widetilde{d\phi_\eps}*\omega) 
\\&=&
(-1)^{n+nk+1+(n-k-1)(k+1)}\int_{\R^n}({d\phi_\eps}*\omega)\wedge\varphi\\&=&
\int_{\R^n} d\omega_\eps\wedge\varphi.
\end{eqnarray*}

\end{proof}

\begin{thm}\label{Stk} 
Let $\omega\in \overline\Omega_{DR}^k(int\ X)$ and $M\subset int\ X$ be a compact sub-manifold with boundary of dimension $(k+1)$. Then, 
$$ \int_M \overline d\omega = \int_{\pa M} \omega \ .$$
\end{thm}  
\begin{proof}
By Lemma \ref{appr_1} we may find a sequence of smooth forms 
$\omega_n$ on $M$ such that $\omega_n\to\omega$ and $d\omega_n\to\overline d \omega$ in $L^\infty$ norm (since $M$ is compact).
So,
$$ \int_M \overline d\omega =\lim_{n\to\infty}\int_M d\omega_n =\lim_{n\to\infty}\int_{\pa M} \omega_n=
\int_{\pa M} \omega \ . $$
\end{proof}

\begin{thm}
In this case 
$$ H^k(\overline\Omega_{DR}^\bullet(int\ X))\cong H^k(\Omega_{DR}^\bullet(int\ X)) \ .$$
\end{thm}
\begin{proof}
We take a sheaf theoretic approach. We will show that the complex of sheaves of germs of forms from 
$\Omega_{DR}^\bullet(int\ X)$ and $\overline\Omega_{DR}^\bullet(int\ X)$ form exact fine resolutions of the sheaf of locally constant functions, which would prove the theorem.
Clearly, both complexes of sheaves are fine. Exactness of $\Omega_{DR}^\bullet(int\ X)$ is the classical Poincar\'e Lemma so we only have to demonstrate exactness of $\overline\Omega_{DR}^\bullet(int\ X)$.
For that matter we introduce some notations.
Let $U$ be some contractible open coordinate chart in $X$ with $\overline U\cap \pa X=\phi$ and $r:U\times I\to U$ be a smooth strong 
deformation retraction to a point $p$ that is, $r(x,0)=p$ and $r(x,1)=x$. Let $\omega\in \overline\Omega_{DR}^k(int\ X)$ be a closed form.
Set 
$$ \gamma:= \int_{t=0}^{t=1} r^*\omega. $$
We have to show that that $\overline d \gamma=\omega$, that is,
$$ \int_U \gamma \wedge d\phi = (-1)^{k}\int_U \omega\wedge\phi\ , $$
where $\phi$ is a smooth $(n-k)$-form with compact support in $U$. 
\begin{eqnarray*}
\int_U \gamma\wedge d\phi = \int_U \left( \int_I r^*\omega \right)\wedge d\phi  &=& 
\int_{U\times I} r^*\omega\wedge d\phi 
\\&=& 
(-1)^k\int_{U\times I} d(r^*\omega\wedge\phi)\\ = (\text{by Theorem \ref{Stk}}) 
&=& 
(-1)^k\left(\int_U r_1^*\omega\wedge\phi - \int_U r_0^*\omega\wedge\phi\right) 
\\&=& 
(-1)^k\int_U \omega\wedge\phi.      
\end{eqnarray*}

\end{proof}

\subsection{Approximation on Manifolds. }$\ $\\
Suppose that $X$ is a compact manifold with boundary $\pa X$ consisting of two
disjoint components $\pa_1 X$ and $\pa_2 X$. Let $U_j$ be a neighborhood of $\pa_j X$ for $j=1,2$.
\begin{prop}\label{prop 2.6.1}
Suppose that we have the data $(X,\pa X,U_1,U_2,\omega)$ where $\omega\in \overline \Omega^k_{DR}(int\ X)$ is $\C^\infty$ in $U_1$ and $\overline d \omega$ is
$\C^\infty$ in $X$. Then for any $\eps>0$ there exists a form $\psi_\eps\in \overline \Omega^k_{DR}(X)$
such that $\| \psi_\eps\|<\eps$, $\|d\psi_\eps\|<\eps$, $\omega+\overline d\psi_\eps$ is $\C^\infty$ on $X-U_2$
and $supp\ \psi_\eps$ does not intersect $\pa X$.
\end{prop}

The proof is the same as the proof of Proposition 2.6.1 in [Y]. \\

\begin{prop}\label{prop 2.6.2}
Let $S\subset\R^n$ be a Riemannian manifold with compact closure. Suppose that $\omega\in \overline \Omega^k_{DR}(S)$ with  $\overline d\omega$ a $\C^\infty$ form. Then, for any positive lower semi-continuous function $\delta:S\to\R_+$
with $\delta(x)\to 0$ as $x$ approaches any point in $\pa S$, there exists a form $\psi\in \overline \Omega^k_{DR}(S)$ such that $|\psi(x) |< \delta(x)$, $|\overline d \psi (x)| < \delta(x)$ and
$\omega+\overline d\psi$ is $\C^\infty$.
\end{prop}
\begin{proof}
The proof is essentially taken from [Y] Proposition 2.6.2 with a modification to accommodate the
pointwise estimates for $|\psi(x)|$ and $|\overline d\psi(x)|$. 
Let $f:S\to\R$ be a $\C^\infty$ function such that $f^{-1}((-\infty,c])$ is compact for any $c\in\R$.
There exists an increasing unbounded sequence $c_1,c_2,\dots\in\R$ such that $f^{-1}(c_i)$ is a smooth
sub-manifold of $S$ for any $i=1,2,\dots\ $. 

Let $Y_i:=f^{-1}([c_{i-2},c_{i+1}])$ be a smooth compact manifold with boundary 
$\pa Y_i= f^{-1}(c_{i-2})\cup f^{-1}(c_{i+1})$, where $c_0=c_{-1}=-\infty$.
Note that $int\ Y_i=Y_i-\pa Y_i=f^{-1}((c_{i-1},c_{i+1}))$.
Set 
$$ \delta_i:=\frac{1}{3}\min \{\delta(x): x\in Y_i\cup Y_{i+1}\cup Y_{i+2}\}\ .  $$
We construct $\psi$ as $\psi=\xi_1+\xi_2+...$ where $\xi_i\in \overline\Omega^{k-1}_{DR}(S) $, 
$supp\ \xi_i\subset int\ Y_i$, $\|\xi_i\|<\delta_i$, $\|\overline d \xi_i\|<\delta_i$. 
Suppose for a moment that the forms $\xi_i$  were constructed for all $i\in\N$. Then, for any $x\in S$ there exists $i$ such that $x\in Y_{i}\cap Y_{i+1}\cap Y_{i+2}$ and $x\notin Y_j$ for $j\in\N-\{i,i+1,i+2\}$ so,
$$ |\psi(x)|=\|\xi_i+\xi_{i+1}+\xi_{i+2}\| < \delta_{i}+\delta_{i+1}+\delta_{i+2} < \delta(x)\ , $$
and
$$ |\overline d\psi(x)|=\|\overline d\xi_i+\overline d\xi_{i+1}+\overline d\xi_{i+2}\| < \delta_{i}+\delta_{i+1}+\delta_{i+2} < \delta(x)\ , $$

The rest of the proof is exactly as in [Y]. The forms $\xi_i$ are constructed inductively such that\\ $\phi_i:=\omega+\overline d(\xi_1+\xi_2+\cdots+\xi_i)$
is $\C^\infty$ on $f^{-1}((-\infty,c_i])$. For $i=0$ set $\phi_0:=\omega$ which is smooth on $f^{-1}(-\infty,c_0])=\phi$. So, we have to construct $\xi_{i+1}$ so that $\phi_i+\overline d \xi_{i+1}$ \
is $\C^\infty$ on $f^{-1}((c_{i+1},c_{i+2}])$.

Apply Proposition \ref{prop 2.6.1} to the data $(Y_{i+1},\pa Y_{i+1},U_1,U_2, \phi_i|_{Y_{i+1}})$, where $\pa_1 Y_{i+1}=f^{-1}(c_{i-1})\ $ and $\ \pa_2 Y_{i+1}=f^{-1}(c_{i+2})$, $\ U_1:=f^{-1}([c_{i-1},c_i))\ $ and $\ U_2:=f^{-1}((c_{i+1},c_{i+2}])$. We note that $\phi_i$ is $\C^\infty$ on $U_1$ by induction hypothesis and $\overline d\phi_i$ is $\C^\infty$ on $Y_{i+1}$. Proposition \ref{prop 2.6.1} provides us with 
a form $\psi_{\delta_{i+1}}\in\overline\Omega^{k-1}_{DR}(Y_{i+1})$ such that 
$\|\psi_{\delta_{i+1}} \|<\delta_{i+1}$, $\|\overline d\psi_{\delta_{i+1}} \|<\delta_{i+1}$,
$\phi_i+\overline d\psi_{\delta_{i+1}}$ is $\C^\infty$ on $Y_{i+1}-U_2=f^{-1}([c_{i-1},c_{i+1}])$ and 
$supp\  \psi_{\delta_{i+1}}$ is a compact subset of $f^{-1}((c_{i-1},c_{i+1}))$.
The latter property shows that $\psi_{\delta_{i+1}}$ can be extended by $0$ to all of $S$; denote by
$\xi_{i+1}$ this extension. 
So, 
$$\|\xi_{i+1}\|=\|\psi_{\delta_{i+1}}\|<\delta_{i+1}\ , \ \|\overline \xi_{i+1}\|= \|\overline d\psi_{\delta_{i+1}}\|<\delta_{i+1}\ .$$
\end{proof}

\subsection{Approximation on stratified sets.}\label{approx_strat}$\ $\\
Let $(X,\Sigma)\subset\R^n$ be a stratified set. 
We will need a notion of a {\bf tubular neighborhood} of a stratum $S\in\Sigma^k$.
The stratum $S\subset\R^n$ is a smooth sub-manifold so 
by [BCR] Corollary 8.9.6,  there exists a neighborhood $U$ of $S$ in $\R^n$ and a smooth retraction
$\pi:U\to S$ such that $d(x,S)=d(x,\pi(x))$ for every $x\in U$.

For each smooth function $\rho_S:S\to\R_+$, $\rho_S(x)\to 0$ as $x$ approaches $\pa S$ we associate a {\bf tubular neighborhood}
of $S$ in $U$: 
$$N_{\rho_S}(S) := \{ x\in U: d(x,S) < \rho_S(\pi(x))\}.$$

Let $\phi_S:N_{\rho_S}(S)\to\R$ be a continuous function and smooth on every stratum that intersects $N_{\rho_S}(S)$, identically equal to $1$ on $\{x:d(x,S)<1/2{\rho_S}(\pi(x))\}$ and $0$ away from $\{x:d(x,S)<3/4{\rho_S}(\pi(x))\}$.

\begin{lem}\label{gamma_ext}
Let $\gamma$ be a form on a stratum $S\in\Sigma$ such that
\begin{equation}\label{bdgm} 
|\gamma(x)|\lesssim\sup \{(1+|d\phi_S(y)|)^{-2}:\pi(y)=x \}, \text{  } x\in S. 
\end{equation}
Then, there exists an $L^\infty$ form $\hat \gamma$ on $X$ such that $\hat \gamma|_S = \gamma$.
\end{lem}
\begin{proof}
Set $\hat \gamma$ to be the extension by $0$ to $X-N_{\rho_S}(S)$ of $\phi_S\pi^*\gamma$. 
To see that $\hat \gamma$ is $L^\infty$ we only have to check that $d\hat\gamma$ is a bounded stratified form.
$$ d\hat \gamma= d\phi_S\wedge\pi^*\gamma+\phi_S d\pi^*\gamma\ . $$
The second summand is clearly bounded and stratified. For the first summand we have
$ |(d\phi_S\wedge \pi^*\gamma)(x)|\to 0 $ as $x$ approaches $\pa S$ due to (\ref{bdgm}).
\end{proof}

Next we prove Theorem \ref{Poincare}. 
\medskip
%
%
%
%
%
%
%
%
\\
\textit{Proof of Theorem  \ref{Poincare}.}
We will prove the following statement.
Let $\omega$ be a closed $L^\infty$ $k$-form on $X$, $p\in X$.
There exists a neighborhood $U$ of $p$ in $X$ and a stratification $\Sigma$ of $U$, 
such that $(\omega,\Sigma)$ is a stratified form, and there exist finitely many $L^\infty$ 
forms on $U$, $\gamma^0,\dots,\gamma^n$, that 
satisfy the following properties.
\begin{enumerate}
\item $\overline d\gamma^j=\omega$ for all $j$.
\item $\gamma^j_l:=\gamma^j|_{\Sigma^{l}}$ is smooth for all $l\leq j$.
\end{enumerate}
Clearly, Theorem \ref{Poincare} follows by setting $\gamma:=\gamma^n$.
%
%
%
%
Set $\Sigma$ to be the stratification constructed in Lemma \ref{weakPoincare} and
set $\gamma^0$ to be the form provided by Lemma \ref{weakPoincare}. Suppose that $\gamma^1,\dots,\gamma^{k-1}$ were constructed and we have to 
construct $\gamma^k$. 
Denote by $N(\Sigma^k)$ the collection of tubular neighborhoods of strata in $\Sigma^k$.
By an appropriate choice of functions $\rho_S$ we may assume that the tubular neighborhoods in $N(\Sigma^k)$
are pairwise disjoint.
As $\gamma^{k-1}_k$ may not be smooth, by Proposition \ref{prop 2.6.2}, we may find forms $\psi^k_S$ on each $S\in\Sigma^k$ such that 
$\gamma^{k-1}_k+\overline d\psi^k_S$ is smooth and 
$$|\psi^k_S(x)|\leq
\sup\{(1+|d\phi_{S}(y)|)^{-2}: \pi(y)=x\}. $$ 
By Lemma \ref{gamma_ext} there exists an extension $\hat\psi^k_S$ of $\psi^k_S$ to $X$.
Since the tubular neighborhoods $N(\Sigma^k)$ do not intersect each other, there exists an $L^\infty$ form $\hat\psi^k$ on $X$ that extends all the $\hat\psi^k_S$ for $S\in\Sigma^k$.
Set  
\begin{eqnarray}
\gamma^k_j&:=&\gamma^{k-1}_j\ \ \ \ \ \ \ \ \ \ \  \text { if } j<k, \nonumber\\
\gamma^k_j&:=&\gamma^{k-1}_{j}+\overline d\hat\psi^k \ \ \text{ if } j\geq k. 
\end{eqnarray}
$\square$


%
%

\end{document}